%% file: 9.19.tex
\documentclass[11pt,a4paper]{article}

\usepackage[a4paper,top=25mm,bottom=27mm,left=28mm,right=28mm,headheight=14pt]{geometry}
\usepackage[T1]{fontenc}
\usepackage[utf8]{inputenc}
\usepackage{amsmath,amssymb,amsthm,mathtools,bm}
\usepackage{mathpazo,mathrsfs}
\usepackage{microtype}
\usepackage{booktabs,array,longtable,tabularx,adjustbox,graphicx,float}
\usepackage{enumitem}
\usepackage{xcolor}
\usepackage{xurl}
\usepackage{hyperref}
\usepackage[noabbrev,capitalise]{cleveref}
\usepackage{fancyhdr}
\usepackage{etoolbox}
\usepackage{aliascnt} 

\usepackage{pgfplots}
\pgfplotsset{compat=1.18}
\usepgfplotslibrary{groupplots}

\usepackage{tikz}
\usetikzlibrary{arrows.meta,positioning}

\hypersetup{
  colorlinks=true,
  linkcolor=black,
  citecolor=black,
  urlcolor=black,
  linktoc=all,
  pdftitle={Uniqueness and Nondegeneracy of Ground States for the L2-Critical Boson Star Equation},
  pdfauthor={Pan Chen, Qi Guo, Juncheng Wei, Yuanyang Yu}
}

\fancypagestyle{plain}{%
  \fancyhf{}
  \fancyfoot[C]{\normalcolor\small\thepage}
  }

\numberwithin{equation}{section}
\newtheorem{theorem}{Theorem}[section]
\newaliascnt{proposition}{theorem}
\newtheorem{proposition}[proposition]{Proposition}
\aliascntresetthe{proposition}
\newaliascnt{lemma}{theorem}
\newtheorem{lemma}[lemma]{Lemma}
\aliascntresetthe{lemma}
\newaliascnt{corollary}{theorem}
\newtheorem{corollary}[corollary]{Corollary}
\aliascntresetthe{corollary}
\newtheorem*{liebyauconjecture}{Lieb-Yau conjecture}
\theoremstyle{definition}
\newaliascnt{definition}{theorem}

\aliascntresetthe{definition}

\newtheoremstyle{italicremark}
  {\topsep}
  {\topsep}
  {\itshape}
  {}
  {\itshape}
  {.}
  {0.5em}
  {}
\theoremstyle{italicremark}

\crefname{theorem}{Theorem}{Theorems}
\crefname{proposition}{Proposition}{Propositions}
\crefname{lemma}{Lemma}{Lemmas}
\crefname{corollary}{Corollary}{Corollaries}
\crefname{definition}{Definition}{Definitions}
 \crefname{remark}{Remark}{Remarks}
\crefname{equation}{equation}{equations}
\crefname{section}{Section}{Sections}
\crefname{subsection}{Section}{Sections}
\crefname{subsubsection}{Section}{Sections}
\crefname{appendix}{Appendix}{Appendices}

\newcommand{\Id}{\mathrm{Id}}

\newcommand{\dd}{\,\mathrm d}

\allowdisplaybreaks[1]
\newif\ifshowrevisions
\showrevisionstrue

\newif\ifshowreviewchanges
\showreviewchangestrue

\title{\bfseries Uniqueness and Nondegeneracy of Ground State for the
$L^2$-Critical Boson Star Equation}
\author{Pan Chen, Qi Guo, Juncheng Wei, Yuanyang Yu}
\date{}

\begin{document}
\maketitle
\vspace{-1.25em}

\begin{abstract}
We prove uniqueness, up to phase rotations and translations, of the ground
state for the $L^2$-critical boson star equation
\[
  (\sqrt{-\Delta}+1)Q
  =\bigl(|x|^{-1}* |Q|^2\bigr)Q,
  \qquad x\in\mathbb R^3.
\]
We also determine the kernels of the real and imaginary linearized
operators.  As a consequence, we show the Lieb-Yau uniqueness conjecture holds when the mass is sufficiently
close to the Chandrasekhar mass.
\end{abstract}

\noindent\textbf{Keywords.} Boson star equation; ground states; uniqueness; nondegeneracy; pseudorelativistic Hartree equation.

\smallskip

\medskip
\tableofcontents

\clearpage
\section{Introduction and main results}\label{1.0}

In this paper, we study ground states of the $L^2$-critical boson star equation
\begin{equation}\label{1.1}
  (\sqrt{-\Delta}+1)Q
  =\bigl(|x|^{-1}*|Q|^2\bigr)Q,
  \qquad Q: \mathbb R^3\rightarrow \mathbb{C}.
\end{equation}
The main results are the uniqueness of ground states up to phase rotations and translations, and the nondegeneracy of the linearized operators.

Equation~\eqref{1.1} is the stationary massless boson star equation
arising in the mean-field description of self-gravitating relativistic
bosons.  More precisely, if $m\geq0$ denotes the rest mass, the boson star
evolution equation is
\[
  i\partial_t\psi
  =\bigl(\sqrt{-\Delta+m^2}-m\bigr)\psi
   -\bigl(|x|^{-1}*|\psi|^2\bigr)\psi.
\]
When $m=0$, the solitary wave $\psi(t,x)=e^{it}Q(x)$ is equivalent to
\eqref{1.1}.  The evolution equation is a mean-field model for
relativistic bosons interacting through their self-gravity
\cite{LiebYau1987,ElgartSchlein2007}.  Its solitary waves, well-posedness, and
collapse dynamics have been studied extensively; see, for example,
\cite{FrohlichJonssonLenzmann2007,FrohlichLenzmann2007,Lenzmann2007,
HerrLenzmann2014,LenzmannLewin2011}.

Our first result gives a classification of all ground states of
\eqref{1.1}.

\begin{theorem} \label{1.2}
The positive radial ground state of \eqref{1.1} is unique; we denote
it by $Q$.  Moreover, every ground state of \eqref{1.1} is of the
form
\[
  e^{i\vartheta}Q(\,\cdot-x_0),
  \qquad
  \vartheta\in\mathbb R,\quad x_0\in\mathbb R^3.
\]
\end{theorem}

For the linearization at $Q$, set
\[
  V:=|x|^{-1}*Q^2.
\]
The real and imaginary components of the linearization are governed,
respectively, by
\begin{align*}
  L_+h
  &:=\bigl(\sqrt{-\Delta}+1-V\bigr)h
     -2Q\bigl(|x|^{-1}*(Qh)\bigr),\\
  L_-h
  &:=\bigl(\sqrt{-\Delta}+1-V\bigr)h.
\end{align*}

\begin{theorem} 
\label{1.3}
Let $Q$ be the positive radial ground state of \eqref{1.1}. Then
\[
  \ker L_+
  =\operatorname{span}_{\mathbb R}
    \{\partial_{x_1}Q,\partial_{x_2}Q,\partial_{x_3}Q\},
  \qquad
  \ker L_- = \operatorname{span}_{\mathbb R}\{Q\}.
\]
\end{theorem}

Thus the only zero modes are those generated by translations and phase
rotations. 
The significance of the critical equation extends beyond the massless model.
It also governs the rescaled concentration profile of the Lieb-Yau uniqueness problem as the particle number
approaches its critical value.  Define the Coulomb-Hartree interaction by
\[
  \mathcal H(u)
  :=\iint_{\mathbb R^3\times\mathbb R^3}
  \frac{|u(x)|^2|u(y)|^2}{|x-y|}\,\mathrm dx\,\mathrm dy.
\]
For a rest mass $m>0$ and a prescribed particle number $N>0$, set
\[
  \begin{aligned}
    \mathcal E_m(u)
    &:=\left\langle
       u,\bigl(\sqrt{-\Delta+m^2}-m\bigr)u
       \right\rangle
       -\frac12\mathcal H(u),\\
    e_m(N)
    &:=\inf_{\|u\|_2^2=N}\mathcal E_m(u).
  \end{aligned}
\]
The Chandrasekhar mass is determined by the sharp inequality
\[
  N_*:=2\inf_{u\ne0}
  \frac{\|u\|_2^2\|(-\Delta)^{1/4}u\|_2^2}{\mathcal H(u)}.
\]
Lieb and Yau proved that, for every $m>0$ and $0<N<N_*$, the problem
$e_m(N)$ admits a positive minimizer and formulated the following conjecture
\cite{LiebYau1987}.

\begin{liebyauconjecture}
For every $m>0$ and every $0<N<N_*$, the variational problem $e_m(N)$ has a
unique positive minimizer up to translations.
\end{liebyauconjecture}

For sufficiently small particle number $N$, Lenzmann proved uniqueness except
possibly for an at most countable set of values of $N$
\cite{Lenzmann2009}.  Guo and Zeng later removed this exceptional set
\cite{GuoZeng2020}.  The remaining difficulty is the intermediate-$N$
regime, and in particular the behavior near the endpoint $N_*$.  As
$N\to N_*$, the concentration theorem of Guo and Zeng places every
positive minimizer, after centering and rescaling, near an optimizer of the
critical problem \cite{GuoZeng2017}.  \Cref{1.2} identifies
this limiting profile, while \cref{1.3} supplies the radial
invertibility needed to analyze the local branch.  This yields the following
near-critical case of the conjecture.

\begin{corollary}
\label{1.4}
There exists $\delta_*\in (0,N_*)$ such that, for every $m>0$ and every $N$ satisfying
\[
  N_* - \delta_* < N < N_*,
\]
the variational problem $e_m(N)$ has a unique positive minimizer up to
translations.
\end{corollary}
The same constant $\delta_*$ works for all rest masses $m>0$.  Together with
the small-$N$ theorem of \cite{GuoZeng2020},
\cref{1.4} leaves only the intermediate-$N$ regime open.
The proof of \cref{1.2} is global rather than
perturbative.  Both the kinetic operator $\sqrt{-\Delta}$ and the Hartree
interaction are nonlocal, so the ODE shooting, comparison, and
Sturm-Liouville arguments available for local radial equations do not apply.
The harmonic extension realizes $\sqrt{-\Delta}$ as a local operator in one
additional dimension, but ordering two boundary traces only up to their first
crossing does not control the nodal domains of their extensions in the upper
half-space.  This is the obstruction behind the gap discussed in
\cite{FrankLenzmann2010}.  Nonrelativistic-limit and implicit-function
arguments give only local uniqueness
\cite{Lenzmann2009}, while the Hartree term prevents a direct application of
the methods developed for fractional Schr\"odinger equations
\cite{FrankLenzmannSilvestre2016}.
 
\paragraph{Main ideas of the proof.}
The nonlocal kinetic and Hartree terms prevent a direct radial ODE
comparison. We therefore use a global variational argument. The proof
addresses the following difficulties.

\begin{enumerate}[label=(\arabic*),leftmargin=*,itemsep=5pt,topsep=5pt]

\item \emph{Choosing a test function.}
We need both a small error term and a coercive quadratic form
on the constrained subspace. Our initial approach used Galerkin
approximations and numerical calculations using \textit{Mathematica}.
A trial function based on a polynomial of degree~16 led to a lengthy
argument. We later chose a rational function $f_0$ associated with the
Benjamin--Ono linearized operator. Its sine transform and Newton
potential are explicit. Jacobi expansions then reduce the required
quadratic-form comparisons to band-matrix estimates. This choice allows
us to prove both coercivity and a small error bound. The computations
guided the construction, but all estimates below are proved analytically.

\item \emph{Comparing all maximizers with the test function.}
The estimates at $f_0$ are not enough: local uniqueness does not exclude
other ground states. After radial reduction, the quotient is invariant
under amplitude changes and mass-preserving dilations. We use these
invariances to choose a representative of each nonnegative maximizer.
We first maximize its $L^2$ inner product with $f_0$ over all dilations.
An amplitude normalization then gives $g=f_0+w$, where $w$ satisfies two
orthogonality conditions. These conditions alone do not control the
size of $w$. However, global maximality also gives inequalities at every
competing scale. Positivity and Newton's formula convert them into a
bound for the mixed potential generated by $f_0g$, without assuming
that $w$ is small.

\item \emph{Controlling the cubic and quartic terms before localization.}
Without smallness, these terms cannot be treated as small remainders.
We instead use two cancellations. Along $f_0+tw$, the numerator minus
the maximal value of the quotient times the denominator is a quartic
polynomial. Its value and derivative vanish at $t=1$. Subtracting one
quarter of the derivative identity from the value identity eliminates
the full fourth-order coefficient, including the contribution from
the denominator. Equivalently, we test the Euler equation of the
scale-invariant quotient at $g$ against $f_0$. One factor is then fixed,
so the expansion has degree at most three in $w$. The quartic term is
retained in a second exact identity.

Next, we keep the sign of the cubic Hartree term and combine it with
the quadratic form. Since $f_0w=f_0g-f_0^2$, the common direct-potential
term cancels in their difference. The remaining mixed potential is
exactly the one bounded in the preceding step. Cauchy--Schwarz also
gives a lower bound for the quartic Hartree term in terms of the cubic
term, while the variational inequality gives an upper bound. Together
with coercivity, the error estimate, and the second exact identity,
these bounds place every such perturbation in the same small ball.
Bounds on the higher variations then give uniform strict concavity
on this ball and hence uniqueness. Using the relations between the
cubic and quartic terms, rather than separate sharp estimates, makes
the earlier argument much shorter.

\item \emph{Excluding additional zero modes.}
Uniqueness alone does not imply nondegeneracy. We therefore return to
the Hessian estimate at a maximizer. It shows that amplitude and
dilation give the only radial zero directions of the scale-invariant
quotient. Comparison with the fixed-frequency quotient removes the
dilation direction. The variational constraint then removes amplitude.
It follows that $L_+$ has no nonzero radial zero mode. Spherical
harmonics and positivity-improving semigroups complete the
identification of the kernels of $L_+$ and $L_-$.

\item \emph{Passing from local uniqueness to uniqueness at fixed mass.}
The critical uniqueness result classifies all optimizers of the sharp
inequality defining the Chandrasekhar mass. As the mass approaches
this value, concentration \cite{GuoZeng2017} gives convergence to a
critical optimizer. The classification identifies the centered and
normalized limit as $Q$. Radial invertibility then gives a unique
local solution branch containing all near-critical rescaled minimizers.
This is not yet uniqueness at fixed mass, since different parameters
could have the same mass. A dilation identity gives strict monotonicity
of the mass along the branch. Equal masses therefore give equal
parameters and the same minimizer. This proves the Lieb--Yau
uniqueness conjecture for masses sufficiently close to the critical
mass from below.

\end{enumerate}

\paragraph{Organization of the paper.}
\Cref{2.0} gives the variational characterization and radial
reduction.  \Cref{3.0} introduces the scale-invariant quotient
and the centre defect.  \Cref{4.0} contains the Jacobi
identities and all quadratic estimates.  \Cref{5.0}
proves the global mixed-potential bound, and
\cref{6.0} proves localization and uniqueness.
\Cref{7.0} treats the linearized kernels;
\cref{8.0} proves the near-critical consequence.

\section{Preliminaries}\label{2.0}

We recall the variational characterization of ground states and the radial half-line formulation.
Unless an ambient space is specified, $\langle\cdot,\cdot\rangle$ is the
$L^2$ pairing and $\|\cdot\|_p$ is the $L^p$ norm on the underlying
Euclidean space or half-line.  The same brackets denote duality when a
factor belongs to a dual form space.  The subscript $\mathrm{rad}$ denotes
radial functions, with the inherited Sobolev norm, and
$\sigma_{\mathrm{ess}}$ denotes the essential spectrum.
We write $\mathrm D^j F$ for the $j$th real Fr\'echet derivative of a map
$F$; thus $F'=\mathrm D F$.  The italic symbol $D$ below is reserved for
the half-line kinetic operator.  The letter $C$ denotes a positive
constant whose value may change between estimates.

We use the unitary Fourier transform
\[
  \widehat u(\xi):=(2\pi)^{-3/2}
  \int_{\mathbb R^3}e^{-ix\cdot\xi}u(x)\,\dd x
\]
and define \(\sqrt{-\Delta}\) by
\[
  \widehat{\sqrt{-\Delta}\,u}(\xi):=|\xi|\widehat u(\xi).
\]
Its quadratic-form domain is \(H^{1/2}(\mathbb R^3)\).  We write
\(\mathcal D(T)\) for the domain of \(T\).  If \(\varphi\) is
measurable, define the maximal multiplication operator by
\[
 (M_\varphi u)(x):=\varphi(x)u(x).
\]
In the following, we set
\[
  M(u):=\|u\|_2^2,\qquad
  T(u):=\|(-\Delta)^{1/4}u\|_2^2,\qquad
  \mathcal K(u):=T(u)+M(u).
\]
The mass-preserving dilation
\[
  u_\lambda(x):=\lambda^{3/2}u(\lambda x),\qquad \lambda>0,
\]
satisfies
\[
  M(u_\lambda)=M(u),\qquad
  T(u_\lambda)=\lambda T(u),\qquad
  \mathcal H(u_\lambda)=\lambda\mathcal H(u).
\]
The action functional associated with \eqref{1.1} is
\[
  \mathcal S(u):=\frac12\mathcal K(u)-\frac14\mathcal H(u).
\]
Writing
\[
  \mathscr N:=\{u\in H^{1/2}(\mathbb R^3)\setminus\{0\}:
                   \mathcal K(u)=\mathcal H(u)\},
\]
a ground state is defined by a minimizer of \(\mathcal S\) on \(\mathscr N\).  For every \(u\ne0\), the unique positive multiple \(tu\) in
\(\mathscr N\) is determined by
\[
  t^2=\frac{\mathcal K(u)}{\mathcal H(u)},\qquad
  \mathcal S(tu)=\frac14\frac{\mathcal K(u)^2}{\mathcal H(u)}.
\]
Hence ground-state directions minimize \(\mathcal K(u)^2/\mathcal H(u)\),
or equivalently maximize \(\mathcal H(u)\) under \(\mathcal K(u)=1\).

We also use the scale-invariant quotient
\[
  \mathcal J(u):=\frac{M(u)T(u)}{\mathcal H(u)},\qquad
  \mathfrak j:=\inf_{u\ne0}\mathcal J(u),\qquad N_*:=2\mathfrak j.
\]
The standard variational argument \cite{LiebYau1987,FrankLenzmann2010} shows that
\(0<\mathfrak j<\infty\) and that the infimum is attained.  Moreover, if
\(w\) is a complex-valued optimizer and \(w^*\) is the Schwarz
rearrangement of \(|w|\), then
\[
  w(x)=e^{i\vartheta}w^*(x-x_0)
\]
for some \(x_0\in\mathbb R^3\) and \(\vartheta\in\mathbb R\).
The two variational problems are equivalent after normalization.  Indeed,
for every \(u\ne0\),
\[
  \inf_{\lambda>0}
  \frac{\mathcal K(u_\lambda)^2}{\mathcal H(u_\lambda)}
  =4\mathcal J(u),
\]
and the minimum is attained when \(T(u_\lambda)=M(u_\lambda)\).
Multiplying the resulting function by the unique positive constant for
which \(\mathcal K=\mathcal H\) gives a minimizer of \(\mathcal S\) on
\(\mathscr N\).  Conversely, every such minimizer yields a minimizer of
\(\mathcal J\).

For the radial reduction, let \(\mathbb R_+:=(0,\infty)\).
Throughout the half-line analysis below, all Hilbert spaces are real.
Let
\[
  \mathscr H_+:=L^2(\mathbb R_+;\mathbb R).
\]
The unitary sine transform on \(\mathscr H_+\) is
\[
  (\mathcal F_sf)(k):=\sqrt{\frac2\pi}
  \int_0^\infty\sin(kr)f(r)\,\dd r.
\]
Set
\[
  D:=\mathcal F_s^{-1}M_k\mathcal F_s,\qquad
  A:=D+\Id,\qquad X_A:=\mathcal D(A^{1/2}).
\]
Equivalently,
\[
  D=\bigl(-\partial_r^2\bigr)_{\rm D}^{1/2},
\]
where \((-\partial_r^2)_{\rm D}\) is the Dirichlet Laplacian on
\(\mathbb R_+\).  Under the radial unitary transformation below, \(D\)
is precisely the half-line representation of \(\sqrt{-\Delta}\) on
radial functions.
For \(f,h\in X_A\), write
\[
  A[f,h]:=\langle A^{1/2}f,A^{1/2}h\rangle,
  \qquad A[f]:=A[f,f].
\]
For any symmetric bilinear form $B$, write $B[f]:=B[f,f]$.
The same notation is used for self-adjoint operators through their
associated forms, on the common form domain under consideration.
Let \(X_A^*\) be the completion of \(\mathscr H_+\) under the norm
\[
  \|\ell\|_{X_A^*}:=\|A^{-1/2}\ell\|_{L^2(\mathbb R_+)},
\]
so that \(X_A^*\) is the dual space of \(X_A\) in the Hilbert scale
associated with \(A\).
We regard \(Af\in X_A^*\) through
\[
  \langle Af,h\rangle_{X_A^*,X_A}:=A[f,h].
\]
For $f,h\in X_A$, the kinetic pairings are likewise understood as
\[
 \langle Df,h\rangle_{X_A^*,X_A}
 :=\langle D^{1/2}f,D^{1/2}h\rangle,
 \qquad \langle f,Df\rangle:=\|D^{1/2}f\|_2^2.
\]
They are $L^2$ operator pairings only when the corresponding vector
belongs to the operator domain.
For a closed subspace \(Y\) of a Hilbert space, \(P_Y\) denotes the
orthogonal projection onto \(Y\).  For a subspace \(Y\subset X_A\), set
\[
  Y^{\perp_A}:=\{h\in X_A:A[y,h]=0\text{ for every }y\in Y\}.
\]
We write \(Y\oplus_A Z\) for the \(A\)-orthogonal direct sum of \(Y\) and \(Z\); for a single vector, \(f^{\perp_A}:=\{f\}^{\perp_A}\).
For quadratic forms \(S\) and \(T\) with a common domain, we use the form-order
notation
\begin{align*}
 S\preceq T
 &\quad\Longleftrightarrow\quad
 (T-S)[h]\geq0\quad\text{for every }h,\\
 S\prec T
 &\quad\Longleftrightarrow\quad
 (T-S)[h]>0\quad\text{for every }h\ne0.
\end{align*}
The reversed symbols \(\succeq\) and \(\succ\) have the corresponding
meaning.
Define
\[
  \mathfrak C(\rho,\sigma)
  :=4\pi\iint_{\mathbb R_+^2}
  \frac{\rho(r)\sigma(s)}{\max\{r,s\}}\,\dd r\,\dd s,
  \qquad
  \mathfrak b(f):=\mathfrak C(f^2,f^2),
\]
and
\[
  V_f(r):=4\pi\int_0^\infty
  \frac{f(s)^2}{\max\{r,s\}}\,\dd s,
  \qquad
  (K_fh)(r):=4\pi f(r)\int_0^\infty
  \frac{f(s)h(s)}{\max\{r,s\}}\,\dd s.
\]
For $f,h,k\in X_A$, the exchange form is
\[
 K_f[h,k]:=\mathfrak C(fh,fk),\qquad
 K_f[h]:=K_f[h,h].
\]
Thus $K_f[h,k]=\langle h,K_fk\rangle$ whenever the operator pairing is
available; otherwise the identity is read in the form sense.

\begin{lemma}\label{2.1}
The map
\[
  \mathcal U_{\rm rad}: U(r)\rightarrow  \sqrt{4\pi}\,rU(r)
\]
is unitary from \(L^2_{\rm rad}(\mathbb R^3;\mathbb R)\) onto
\(\mathscr H_+\), and maps
\(H^{1/2}_{\rm rad}(\mathbb R^3;\mathbb R)\) onto \(X_A\).  If \(f_U(r):=rU(r)\), then
\begin{equation*}
  \mathcal K(U)=4\pi A[f_U],\qquad
  \mathcal H(U)=4\pi\mathfrak b(f_U).
\end{equation*}
Moreover, \(\mathfrak b\) is a continuous homogeneous quartic polynomial
on \(X_A\), and
\[
  |\mathfrak b(f)|\le C A[f]^2,
  \qquad
  \mathfrak b'(f)[h]
  =4\int_0^\infty V_f(r)f(r)h(r)\,\dd r.
\]
\end{lemma}

\begin{proof}
For a smooth real radial function \(U\), with \(f_U(r):=rU(r)\), the
radial Fourier formula in our normalization gives
\[
  (\mathcal F_s f_U)(k)=k\widehat U(k).
\]
The unitary profile is $\mathcal U_{\rm rad}U=\sqrt{4\pi}f_U$, whereas
$f_U=rU$ is used in the nonlinear equation.  In particular,
\[
 \|U\|_2^2=4\pi\|f_U\|_2^2,\qquad
 \|(-\Delta)^{1/4}U\|_2^2
 =4\pi\|D^{1/2}f_U\|_2^2.
\]
Hence Plancherel's theorem yields
\[
  \mathcal K(U)
  =4\pi\int_0^\infty(1+k)|(\mathcal F_sf_U)(k)|^2\,\dd k
  =4\pi A[f_U].
\]
This identity also shows, by passage to the closed form domains, that
the radial transformation identifies
\(H^{1/2}_{\rm rad}(\mathbb R^3;\mathbb R)\) with \(X_A\).
Newton's theorem gives
\[
  (|x|^{-1}*U^2)(r)
  =4\pi\int_0^\infty
    \frac{f_U(s)^2}{\max\{r,s\}}\,\dd s
  =V_{f_U}(r),
\]
and therefore
\[
  \mathcal H(U)=4\pi\mathfrak b(f_U).
\]
Finally, the Hardy-Littlewood-Sobolev inequality and
\(H^{1/2}(\mathbb R^3)\hookrightarrow L^{12/5}(\mathbb R^3)\) give
\[
  |\mathfrak b(f)|\le C A[f]^2.
\]
By polarization, \(\mathfrak b\) is a continuous quartic polynomial,
and differentiation gives
\[
  \mathfrak b'(f)[h]
  =4\int_0^\infty V_f(r)f(r)h(r)\,\dd r.
\]
See, for example, \cite{Lenzmann2007} for the corresponding standard
Hartree estimates.
\end{proof}

\begin{proposition}\label{2.2}
One has
\[
  \sup_{0\ne U\in H^{1/2}(\mathbb R^3;\mathbb C)}
  \frac{\mathcal H(U)}{\mathcal K(U)^2}
  =\frac1{4\pi}
  \sup_{\substack{f\in X_A\\A[f]=1}}\mathfrak b(f),
\]
and radial ground-state directions correspond exactly to the maximizers of
\(\mathfrak b\) on the \(A\)-unit sphere.
\end{proposition}

\begin{proof}
We first reduce the three-dimensional variational problem to radial real-valued
functions.  Given \(U\in H^{1/2}(\mathbb R^3;\mathbb C)\), let
\[
  u:=|U|,
  \qquad
  u^*:=\text{the Schwarz rearrangement of }u.
\]
The fractional diamagnetic and P\'olya-Szeg\H{o} inequalities, together
with the Riesz rearrangement inequality, give
\[
  M(u^*)=M(U),\qquad
  T(u^*)\le T(U),\qquad
  \mathcal H(u^*)\ge \mathcal H(U).
\]
See, for example, \cite[Section~2.1]{FrankLenzmann2010} for the fractional
rearrangement inequalities and \cite[Chapter~3]{LiebLoss2001} for the Riesz
rearrangement inequality.  Consequently,
\[
  \frac{\mathcal H(U)}{\mathcal K(U)^2}
  \le
  \frac{\mathcal H(u^*)}{\mathcal K(u^*)^2}.
\]
Since \(u^*\) is nonnegative and radial, while radial functions form a
subclass of \(H^{1/2}(\mathbb R^3;\mathbb C)\), it follows that
\[
  \sup_{0\ne U\in H^{1/2}(\mathbb R^3;\mathbb C)}
  \frac{\mathcal H(U)}{\mathcal K(U)^2}
  =
  \sup_{\substack{0\ne U\in H^{1/2}_{\rm rad}(\mathbb R^3)\\ U\ \mathrm{real}}}
  \frac{\mathcal H(U)}{\mathcal K(U)^2}.
\]
Now let \(U\) be real and radial and set \(f_U(r):=rU(r)\).
By \cref{2.1},
\[
  \mathcal K(U)=4\pi A[f_U],
  \qquad
  \mathcal H(U)=4\pi\mathfrak b(f_U).
\]
Hence
\[
  \frac{\mathcal H(U)}{\mathcal K(U)^2}
  =
  \frac1{4\pi}
  \frac{\mathfrak b(f_U)}{A[f_U]^2}.
\]
The map in \cref{2.1} identifies
\(H^{1/2}_{\rm rad}(\mathbb R^3)\) with \(X_A\).  Therefore
\[
  \sup_{0\ne U\in H^{1/2}(\mathbb R^3;\mathbb C)}
  \frac{\mathcal H(U)}{\mathcal K(U)^2}
  =
  \frac1{4\pi}
  \sup_{0\ne f\in X_A}
  \frac{\mathfrak b(f)}{A[f]^2}.
\]
Since \(A\) is quadratic and \(\mathfrak b\) is homogeneous of degree four,
the quotient on the right is invariant under multiplication of \(f\) by a
nonzero scalar.  Thus, after replacing
\[
  f\longmapsto \frac{f}{\sqrt{A[f]}},
\]
we obtain
\[
  \sup_{0\ne U\in H^{1/2}(\mathbb R^3;\mathbb C)}
  \frac{\mathcal H(U)}{\mathcal K(U)^2}
  =
  \frac1{4\pi}
  \sup_{\substack{f\in X_A\\ A[f]=1}}
  \mathfrak b(f).
\]

Finally, by the variational characterization preceding this proposition,
ground-state directions are exactly the maximizers of
\(\mathcal H(U)/\mathcal K(U)^2\).  The preceding identity shows that, for
radial \(U\), this is equivalent to the \(A\)-normalized profile
\[
  \frac{f_U}{\sqrt{A[f_U]}}
\]
being a maximizer of \(\mathfrak b\) on the \(A\)-unit sphere.  Conversely,
every such maximizer \(f\) determines, through
\cref{2.1}, a radial direction maximizing the original
three-dimensional quotient, and hence a radial ground-state direction.
\end{proof}

\begin{lemma}\label{2.3}
Let \(f\) maximize \(\mathfrak b\) subject to \(A[f]=1\), and set
\(\beta:=\mathfrak b(f)\).  Then
\[
  V_ff=\beta Af\qquad\text{in }X_A^*.
\]
Hence \(F:=\beta^{-1/2}f\) satisfies \(AF=V_FF\).  Conversely, every radial
ground state yields, after \(A\)-normalization, a maximizer of \(\mathfrak
b\) on the \(A\)-unit sphere.
\end{lemma}

\begin{proof}
Since \(\mathfrak b\in C^1(X_A)\) and \(A[f]=1\), the standard
Lagrange-multiplier principle gives some \(\mu\in\mathbb R\) such that
\[
  4V_ff=2\mu Af
  \qquad\text{in }X_A^*.
\]
Pairing with \(f\) and using \(A[f]=1\) gives
\[
  4\beta=2\mu,
\]
hence \(\mu=2\beta\) and
\[
  V_ff=\beta Af.
\]
Since \(f\neq0\), the positivity of the Newton kernel gives
\(\beta=\mathfrak b(f)>0\).  For
\(F:=\beta^{-1/2}f\), one has \(V_F=\beta^{-1}V_f\), and therefore
\[
  V_FF
  =\beta^{-3/2}V_ff
  =\beta^{-1/2}Af
  =AF.
\]
The converse follows from \cref{2.2} and the
variational characterization above.
\end{proof}

By the existence result recalled in Section~\ref{2.0} and
\cref{2.2}, the half-line maximization problem admits a
nonnegative maximizer.  It remains to prove that its maximizing set consists
of a single antipodal pair.

\section{The scale-invariant quotient and its expansion}
\label{3.0}
We introduce a quotient invariant under multiplication by constants and
under dilation, and expand it around a fixed function $f_0$.  These
identities will be used to compare maximizers in \cref{6.0}.

For $f\in X_A$, write
\[
 M(f):=\|f\|_2^2,\qquad T(f):=\|D^{1/2}f\|_2^2,
 \qquad A[f]=M(f)+T(f).
\]
Since $D$ has trivial kernel, $T(f)>0$ for $f\ne0$.  Define
\begin{equation}\label{3.1}
 \mathcal B(f):=\frac{\mathfrak b(f)}{4M(f)T(f)}\quad(f\ne0),
 \qquad \beta_*:=\sup_{A[f]=1}\mathfrak b(f).
\end{equation}
For $\lambda>0$, the dilation $f_\lambda(r):=\lambda^{1/2}f(\lambda r)$
satisfies
\[
 M(f_\lambda)=M(f),\qquad T(f_\lambda)=\lambda T(f),\qquad
 \mathfrak b(f_\lambda)=\lambda\mathfrak b(f).
\]
After normalization, $f_\lambda/\sqrt{A[f_\lambda]}$ lies on $A[f]=1$ and
\[
 \mathfrak b\!\left(\frac{f_\lambda}{\sqrt{A[f_\lambda]}}\right)
 =\frac{\lambda\mathfrak b(f)}{(M(f)+\lambda T(f))^2}
 \le\frac{\mathfrak b(f)}{4M(f)T(f)}.
\]
Equality holds exactly when $\lambda=M(f)/T(f)$.  Thus every value of
$\mathcal B$ is attained by $\mathfrak b$ on the $A$-unit sphere.
Conversely, taking $\lambda=1$ gives $\mathfrak b(f)\le\mathcal B(f)$
when $A[f]=1$.  Together with homogeneity, this proves
\begin{equation}\label{3.2}
 \begin{gathered}
 \mathcal B(af_\lambda)=\mathcal B(f)
 \qquad(a\in\mathbb R\setminus\{0\},\ \lambda>0),\\
 \sup_{\lambda>0}\frac{\mathfrak b(f_\lambda)}{A[f_\lambda]^2}
 =\mathcal B(f),\qquad
 \sup_{f\ne0}\mathcal B(f)=\beta_*.
 \end{gathered}
\end{equation}
If $f$ maximizes $\mathfrak b$ on $A[f]=1$, then
$\mathfrak b(f)=\mathcal B(f)=\beta_*$.  Hence equality holds in the
preceding inequality at $\lambda=1$, so $M(f)=T(f)$.  Since their sum
is $1$, every such maximizer satisfies
\begin{equation}\label{3.3}
 M(f)=T(f)=\frac12.
\end{equation}
This conclusion applies to the normalized maximizers, not to all their
multiples and dilations.

Fix a positive $f_0\in\mathcal D(A)$ normalized by
\[
 M(f_0)=T(f_0)=\frac12,\qquad \beta_0:=\mathfrak b(f_0).
\]
The explicit choice of $f_0$ is given in \cref{4.0}.  We do not assume that $f_0$
is a critical point.  Its error in the Euler equation of \cref{2.3} is
\[
 \varepsilon_0:=V_{f_0}f_0-\beta_0Af_0\in X_A^*.
\]
Since $A[f_0]=1$, we have $\beta_0\le\beta_*$ and
$\langle\varepsilon_0,f_0\rangle=0$.  For the expansion below, define
\[
 H_Z[h]:=\beta_0A[h]-\mathfrak C(f_0^2,h^2)
                    -2\mathfrak C(f_0h,f_0h),\qquad h\in X_A.
\]
The relation of $\varepsilon_0$ and $H_Z$ to the first and second
variations is given at the end of this section.
For products $\rho$ of functions in $X_A$, write
\[
 V[\rho](r):=4\pi\int_0^\infty
                   \frac{\rho(s)}{\max\{r,s\}}\,\dd s.
\]
Thus $V_f=V[f^2]$ and
$\mathfrak C(\rho,h^2)=\int_0^\infty V[\rho](r)h(r)^2\,\dd r$.
We first record two estimates for $\mathfrak C$.

\begin{lemma}\label{3.4}
Let $\rho$ and $\sigma$ be products of two functions in $X_A$.  Then
\[
 \mathfrak C(\rho,\rho)\ge0,\qquad
 |\mathfrak C(\rho,\sigma)|^2
 \le\mathfrak C(\rho,\rho)\mathfrak C(\sigma,\sigma).
\]
Moreover,
\begin{equation}\label{3.6}
 0\le\mathfrak b(h)\le\frac{\pi^2}{2}A[h]^2<5A[h]^2
 \qquad(h\in X_A\setminus\{0\}).
\end{equation}
\end{lemma}
\begin{proof}
The identity $1/\max\{r,s\}=\int_{\max\{r,s\}}^\infty t^{-2}\,\dd t$ gives,
for smooth compactly supported densities,
\[
 \mathfrak C(\rho,\sigma)=4\pi\int_0^\infty\frac1{t^2}
 \left(\int_0^t\rho(r)\,\dd r\right)
 \left(\int_0^t\sigma(s)\,\dd s\right)\dd t.
\]
This proves nonnegativity and Cauchy-Schwarz for signed densities.
The continuous four-linear Hartree form from \cref{2.1} extends these
inequalities to products of $X_A$ functions by approximation.

The radial Hardy-Kato inequality
\cite{LiebLoss2001,FrankSeiringer2008} gives
\[
 \int_0^\infty\frac{h(r)^2}{r}\,\dd r\le\frac\pi2T(h).
\]
Since $\max\{r,s\}\ge\sqrt{rs}$, Cauchy-Schwarz yields
\begin{align*}
 \mathfrak b(h)
 &\le4\pi\left(\int_0^\infty r^{-1/2}h(r)^2\,\dd r\right)^2\\
 &\le4\pi M(h)\int_0^\infty\frac{h(r)^2}{r}\,\dd r
 \le2\pi^2M(h)T(h)
 \le\frac{\pi^2}{2}A[h]^2.
\end{align*}
The last step uses $4M(h)T(h)\le(M(h)+T(h))^2$.  Since $\pi^2/2<5$,
the proof is complete.
\end{proof}

To remove the linear terms in $M(f_0+tw)$ and $T(f_0+tw)$, set
\begin{equation}\label{3.7}
 Y_2:=\{w\in X_A:\langle f_0,w\rangle=\langle Df_0,w\rangle=0\}.
\end{equation}
For the $f_0$ chosen in \cref{4.0}, \cref{5.1} will show that every
nonnegative normalized maximizer can be written as $f_0+w$, with
$w\in Y_2$, after multiplication by a positive constant and a dilation.
The following identities hold for every $w\in Y_2$, without any
smallness assumption.

\begin{lemma}\label{3.8}
For $w\in Y_2$,
\begin{equation}\label{3.9}
 \widetilde\Psi(w):=\mathcal B(f_0+w)
 =\frac{\mathfrak b(f_0+w)}{1+2A[w]+4M(w)T(w)}.
\end{equation}
For every $t\in\mathbb R$,
\begin{equation}\label{3.10}
\begin{split}
 \mathfrak b(f_0+tw)
 ={}&\beta_0+4t\langle\varepsilon_0,w\rangle
     +2t^2\bigl(\beta_0A[w]-H_Z[w]\bigr)\\
 &+4t^3\mathfrak C(f_0w,w^2)+t^4\mathfrak b(w).
\end{split}
\end{equation}
\end{lemma}
\begin{proof}
The two orthogonality conditions give
\[
 M(f_0+tw)=\frac12+t^2M(w),\qquad
 T(f_0+tw)=\frac12+t^2T(w).
\]
In particular, $f_0+tw\ne0$, and multiplication gives
\[
 4M(f_0+tw)T(f_0+tw)=1+2t^2A[w]+4t^4M(w)T(w).
\]
Taking $t=1$ proves \eqref{3.9}.

To expand the numerator, substitute
$(f_0+tw)^2=f_0^2+2tf_0w+t^2w^2$ in both arguments of $\mathfrak C$.
Since $A[f_0,w]=0$, the coefficient of $t$ is
\[
 4\mathfrak C(f_0^2,f_0w)
 =4\langle V_{f_0}f_0,w\rangle
 =4\langle\varepsilon_0,w\rangle.
\]
The coefficient of $t^2$ is
\[
 2\mathfrak C(f_0^2,w^2)+4\mathfrak C(f_0w,f_0w)
 =2\bigl(\beta_0A[w]-H_Z[w]\bigr).
\]
The remaining coefficients are $\beta_0$,
$4\mathfrak C(f_0w,w^2)$, and $\mathfrak b(w)$, proving
\eqref{3.10}.  This is an exact polynomial identity.
\end{proof}

Dividing the expansion by $1+2t^2A[h]+4t^4M(h)T(h)$ shows that, for
$h\in Y_2$,
\[
 \mathrm D\widetilde\Psi(0)[h]=4\langle\varepsilon_0,h\rangle,
 \qquad
 \mathrm D^2\widetilde\Psi(0)[h,h]=-4H_Z[h].
\]
Thus a bound for $\varepsilon_0$ controls the first variation, and
a positive lower bound for $H_Z$ gives a negative second variation at $0$.
The corresponding estimates for our choice of $f_0$ are proved in
\cref{4.0}; the higher-order terms are treated in \cref{6.0}.

Finally, $w\in Y_2$ implies $A[f_0+w]=1+A[w]$, but does not imply
$M(f_0+w)=T(f_0+w)$.  The denominator must therefore be kept in its
exact form:
\[
 1+2A[w]+4M(w)T(w)
 =(1+A[w])^2-(T(w)-M(w))^2.
\]
In general it cannot be replaced by $(1+A[w])^2$.  This distinction
is also used in the Hessian calculation in \cref{7.0}.

\section{The Jacobi-Benjamin-Ono centre and quadratic estimates}
\label{4.0}
The centre is explicit but is not an exact critical point.
The required estimates are proved on the two-constraint space $Y_2$.

\subsection{Jacobi-Laguerre basis and normalization}
We first introduce a basis that will be used to define the test function and to represent the operators in the following estimates. Recall the Jacobi and generalized Laguerre polynomials \cite{Szego1975}, for $n=0,1,\ldots$ and $\alpha,\beta>-1$:
\[
P_n^{(\alpha,\beta)}(t)
=
\frac{(-1)^n}{2^n n!}
(1-t)^{-\alpha}(1+t)^{-\beta}
\frac{d^n}{dt^n}
\left[
  (1-t)^{n+\alpha}(1+t)^{n+\beta}
\right],
\]
\[
L_n^{(\alpha)}(x)
=
\frac{x^{-\alpha}e^x}{n!}
\frac{d^n}{dx^n}
\left(
  e^{-x}x^{n+\alpha}
\right).
\]
Set
\[
\chi(r):=\frac{1-4r^2/9}{1+4r^2/9},
 \]
 and
 \[
 \nu_n:=\int_{-1}^1(1-t)^{1/2}(1+t)^{3/2}
 \bigl[P_n^{(1/2,3/2)}(t)\bigr]^2\,\dd t.
\]
Define
\begin{equation}
 \phi_n(r):=\frac{4\sqrt{2/3}}{\sqrt{\nu_n}}
 \frac{2r/3}{(1+4r^2/9)^2}
 P_n^{(1/2,3/2)}(\chi(r)).
 \label{4.1}
\end{equation}

\begin{lemma}
\label{4.2}
The family $\{\phi_n\}_{n\geq0}$ is an orthonormal basis of
$\mathscr H_+$, and its finite linear span is a form core for $X_A$.
Moreover,
\begin{equation}
 (\mathcal F_s\phi_n)(k)
 =(-1)^n\sqrt{\frac{27}{(n+1)(n+2)}}
 ke^{-3k/2}L_n^{(2)}(3k).
 \label{4.3}
\end{equation}
For a symmetric operator or form $B$, its Jacobi matrix is
$B_{mn}:=B[\phi_m,\phi_n]$, $m,n\ge0$.  In particular,
\begin{equation}
 A_{nn}=2+\frac{2n}{3},
 \qquad
 A_{n,n+1}=\frac{\sqrt{(n+1)(n+3)}}{3},
 \qquad
 A_{mn}=0\quad(|m-n|>1).
 \label{4.4}
\end{equation}

\end{lemma}

\begin{proof}
Make the change of variables
\[
 t=\frac{1-(2r/3)^2}{1+(2r/3)^2},
 \qquad
 r=\frac32\sqrt{\frac{1-t}{1+t}}.
\]
The map
\[
 f(r)\longmapsto
 \sqrt{\frac32}\,(1-t)^{-1/2}(1+t)^{-3/2}
 f\left(\frac32\sqrt{\frac{1-t}{1+t}}\right)
\]
is unitary from $L^2(\mathbb R_+,\dd r)$ onto
\[
 L^2\bigl((-1,1),(1-t)^{1/2}(1+t)^{3/2}\,\dd t\bigr).
\]
Under this map, \(\phi_n\) is sent to
\[
\frac{P_n^{(1/2,3/2)}}{\sqrt{\nu_n}}.
\]
The orthogonality of the Jacobi polynomials proves that \(\{\phi_n\}\) is orthonormal, while the density of polynomials proves completeness. Hence \(\{\phi_n\}_{n\geq0}\) is an orthonormal basis of \(\mathscr H_+\).
Put $a=3/2$.  Taylor expansion of the Jacobi polynomial at $-1$ gives
\begin{align*}
 \phi_n(r)
 &=\frac{4a^{5/2}}{\sqrt{\nu_n}}\frac{r}{(r^2+a^2)^2}
   P_n^{(1/2,3/2)}\!\left(-1+\frac{2a^2}{r^2+a^2}\right)\\
 &=\frac{4a^{5/2}}{\sqrt{\nu_n}}
   \sum_{j=0}^n\frac{(2a^2)^j}{j!}
   \left.\partial_t^jP_n^{(1/2,3/2)}(t)\right|_{t=-1}
   \frac{r}{(r^2+a^2)^{j+2}}.
\end{align*}
For $a,k>0$ and integers $j\ge1$,
\begin{align*}
 \int_0^\infty\frac{\cos(kr)}{r^2+a^2}\,\dd r
 &=\frac{\pi}{2a}e^{-ak},\\
 \int_0^\infty\frac{\cos(kr)}{(r^2+a^2)^{j+1}}\,\dd r
 &=-\frac1{2ja}\partial_a
       \int_0^\infty\frac{\cos(kr)}{(r^2+a^2)^j}\,\dd r,\\
 \int_0^\infty\frac{r\sin(kr)}{(r^2+a^2)^{j+1}}\,\dd r
 &=\frac{k}{2j}
       \int_0^\infty\frac{\cos(kr)}{(r^2+a^2)^j}\,\dd r.
\end{align*}
Differentiation is dominated for $a$ in compact subsets of $(0,\infty)$;
the boundary terms in the last identity vanish at $0$ and $\infty$.
Induction in $j$ shows that $e^{ak}(\mathcal F_s\phi_n)(k)/k$ is a
polynomial of degree at most $n$.  Its leading coefficient is positive:
\begin{align*}
 \frac1{n!}\partial_t^nP_n^{(1/2,3/2)}(t)
 &=2^{-n}\binom{2n+2}{n},\\
 \lim_{k\to\infty}\frac{e^{ak}(\mathcal F_s\phi_n)(k)}{k^{n+1}}
 &=\frac{\sqrt{2\pi}\,a^{n+3/2}}{2^n(n+1)!\sqrt{\nu_n}}
       \binom{2n+2}{n}>0.
\end{align*}
Unitarity of $\mathcal F_s$ makes these degree-$n$ polynomials orthogonal
for the weight $k^2e^{-2ak}\,\dd k$.  Since
\[
 \int_0^\infty k^2e^{-2ak}\bigl[L_n^{(2)}(2ak)\bigr]^2\,\dd k
 =\frac{(n+1)(n+2)}{8a^3},
\]
normalization and the sign of the leading coefficient give
\[
 (\mathcal F_s\phi_n)(k)
 =(-1)^n\sqrt{\frac{8a^3}{(n+1)(n+2)}}
       ke^{-ak}L_n^{(2)}(2ak).
\]
Taking $a=3/2$ proves \eqref{4.3}.
For $n\ge0$, with $L_{-1}^{(2)}:=0$, the generalized Laguerre polynomials satisfy
\[
xL_n^{(2)}(x) = -(n+1)L_{n+1}^{(2)}(x) +(2n+3)L_n^{(2)}(x) -(n+2)L_{n-1}^{(2)}(x). 
\]
Applying this recurrence to \eqref{4.3} and using the normalization constants gives
\[
A_{nn} = 2+\frac{2n}{3}, \qquad A_{n,n+1} = \frac{\sqrt{(n+1)(n+3)}}{3},
\]
while all matrix elements with \(|m-n|>1\) vanish. This proves \eqref{4.4}.
Under the sine transform, the form norm is the norm of
$L^2(\mathbb R_+,(1+k)\,\dd k)$.  Multiplication by $ke^{-3k/2}$
identifies polynomial approximation with density of polynomials in
\[
 L^2\bigl(\mathbb R_+,(1+k)k^2e^{-3k}\,\dd k\bigr).
\]
Indeed, let $u$ be orthogonal to all polynomials in this weighted space.
Then
\[
 \int_0^\infty u(k)k^j(1+k)k^2e^{-3k}\,\dd k=0,
 \qquad j=0,1,\ldots.
\]
For $z\in\mathbb C$ with $|\operatorname{Re}z|<1$, Cauchy-Schwarz gives
\begin{align*}
 &\int_0^\infty |u(k)|(1+k)k^2e^{-3k+\operatorname{Re}z\,k}\,\dd k\\
 &\quad\le
 \|u\|_{L^2((1+k)k^2e^{-3k}\,\dd k)}
 \left(\int_0^\infty(1+k)k^2e^{-(3-2\operatorname{Re}z)k}\,\dd k\right)^{1/2}
 <\infty.
\end{align*}
Thus $\int_0^\infty u(k)(1+k)k^2e^{-3k+zk}\,\dd k$ is analytic on this
strip and all its derivatives at zero vanish.  It vanishes identically;
on the imaginary axis, uniqueness of the Fourier transform implies
$u=0$.  Hence the polynomials are dense in the form norm.
\end{proof}

We now specify the test function used in this paper:
\[
 f_0(r):=\frac{\phi_0(r)}{\sqrt{2}}
 =\frac{4(3/2)^{5/2}}{\sqrt{\pi}}
   \frac{r}{(r^2+9/4)^2},\qquad r>0.
\]

\begin{lemma}\label{4.5}
The function $f_0$ is positive on $\mathbb R_+$, belongs to
$\mathcal D(A)$, and satisfies
\[
 M(f_0)=T(f_0)=\frac12,\qquad
 A[f_0]=1,\qquad
 \beta_0=\mathfrak b(f_0)=\frac73.
\]
\end{lemma}

\begin{proof}
For $a>0$, set
\[
 \varphi_a(r):=\frac{4\sqrt{2}\,a^{5/2}}{\sqrt{\pi}}
              \frac{r}{(r^2+a^2)^2}.
\]
Integration by parts and the identity
\[
 \int_0^\infty\frac{\cos(kr)}{r^2+a^2}\,\dd r
 =\frac{\pi}{2a}e^{-ak}
\]
give
\[
 (\mathcal F_s\varphi_a)(k)=2a^{3/2}ke^{-ak}.
\]
Thus $\varphi_a\in\mathcal D(A)$, and Plancherel's theorem yields
\begin{align*}
 M(\varphi_a)&=4a^3\int_0^\infty k^2e^{-2ak}\,\dd k=1,\\
 T(\varphi_a)&=4a^3\int_0^\infty k^3e^{-2ak}\,\dd k=\frac{3}{2a}.
\end{align*}
Splitting the Newton integral into the regions $r>s$ and $r<s$ gives
\begin{align*}
 \mathfrak b(\varphi_a)
 &=8\pi\int_0^\infty\varphi_a(s)^2
       \left(\int_s^\infty\frac{\varphi_a(r)^2}{r}\,\dd r\right)\dd s\\
 &=\frac{4096a^{10}}{3\pi}
       \int_0^\infty\frac{s^2}{(s^2+a^2)^7}\,\dd s\\
 &=\frac{4096a^{10}}{3\pi}\frac{21\pi}{2048a^{11}}
 =\frac{14}{a}.
\end{align*}
For $a=3/2$, we have $\varphi_a=\phi_0$ and
$f_0=\varphi_a/\sqrt{2}$.  Hence
\[
 M(f_0)=T(f_0)=\frac12,\qquad
 A[f_0]=1,\qquad
 \beta_0=\frac14\mathfrak b(\varphi_{3/2})=\frac73.
\]
Positivity follows directly from the formula for $f_0$.
\end{proof}

\input{boson_f0_figures_latex/fig_f0.tex}

For the explicit centre above and each integer $N\ge2$, set
\begin{equation*}
 Y_N:=\overline{\operatorname{span}}^{\,X_A}\{\phi_n:n\ge N\}
 =\{h\in X_A:\langle h,\phi_j\rangle=0,\ 0\le j<N\}.
\end{equation*}

If $h\in X_A$ has its first $N$ coefficients equal to zero, choose
finite Jacobi sums $h_j\to h$ in $X_A$.  Then
\[
 h_j-\sum_{n=0}^{N-1}\langle h_j,\phi_n\rangle\phi_n
 \longrightarrow h\quad\text{in }X_A,
\]
which proves the equality.
For $N=2$, this agrees with
\eqref{3.7}, since \eqref{4.4} gives
\[
 \operatorname{span}\{f_0,Df_0\}
 =\operatorname{span}\{\phi_0,\phi_1\}.
\]
The half-line dilation generator is defined by
\[
 \mathcal D(\Lambda):=
 \{f\in\mathscr H_+:rf'+\tfrac12f\in\mathscr H_+\},
 \qquad \Lambda f:=rf'+\frac12f,
\]
where the derivative is distributional.  We shall use
\begin{equation}\label{4.6}
 \Lambda\phi_0=\frac{\sqrt3}{2}\phi_1,\qquad
 \Lambda\phi_1=-\frac{\sqrt3}{2}\phi_0+\sqrt2\phi_2.
\end{equation}
Both identities follow by differentiating \eqref{4.1} for
$n=0,1$.  In particular, $\|\Lambda\phi_1\|_2^2=11/4$.

\subsection{The Benjamin-Ono comparison and the exchange matrix}
Set
\begin{equation*}
 \Sigma(r):=\frac{Af_0(r)}{f_0(r)}
 =\frac13+\frac{8}{3(1+x^2)},\qquad
 x:=\frac{2r}{3},\qquad
 L_{\mathrm{BO}}:=\frac32(A-M_\Sigma).
\end{equation*}
Under the unitary rescaling $\widetilde h(x)=\sqrt{3/2}\,h(3x/2)$,
\[
 \widetilde{L_{\mathrm{BO}}h}
 =\left(\left(-\partial_x^2\right)_{\rm D}^{1/2}
       +1-\frac4{1+x^2}\right)\widetilde h,
 \qquad L_{\mathrm{BO}}f_0=0.
\]
The Jacobi recurrences give
\begin{equation*}
 (L_{\mathrm{BO}})_{nn}
 =n+\frac12-\frac1{(n+1)(n+2)},\qquad
 (L_{\mathrm{BO}})_{n,n+1}
 =\frac n2\sqrt{1-\frac1{(n+2)^2}}.
\end{equation*}
All other entries with $|m-n|>1$ vanish.  Similarly, for
$s(r)=(1+4r^2/9)^{-1}$,
\begin{equation}\label{4.7}
 (M_s)_{nn}=\frac12+\frac1{4(n+1)(n+2)},\qquad
 (M_s)_{n,n+1}=\frac{\sqrt{(n+1)(n+3)}}{4(n+2)}.
\end{equation}

\begin{lemma}\label{4.8}
Let $k_{mn}:=\langle\phi_m,K_{f_0}\phi_n\rangle$ for $m,n\ge0$.  For $n\ge1$,
\begin{equation}\label{4.9}
\begin{split}
 k_{nn}&=\frac4{n(n+3)},\\
 k_{n,n+1}&=\frac{8(n+2)}{3[(n+1)(n+3)]^{3/2}},\\
 k_{n,n+2}&=\frac{2\sqrt{(n+1)(n+4)}}
                 {3[(n+2)(n+3)]^{3/2}}.
\end{split}
\end{equation}
For $m,n\ge1$ and $|m-n|>2$, one has $k_{mn}=0$.
\end{lemma}
\begin{proof}
Put
\[
 \varpi(t)=(1-t)^{1/2}(1+t)^{3/2},\quad
 \varpi_+(t)=(1-t)^{3/2}(1+t)^{5/2},\quad
 \widehat P_j=P_j^{(3/2,5/2)}.
\]
For $G_n(r)=\int_0^r f_0\phi_n$ and $t=\chi(r)$, Newton's formula and the
Jacobi primitive identity give
\begin{align*}
 k_{mn}&=4\pi\int_0^\infty\frac{G_m(r)G_n(r)}{r^2}\,\dd r,\\
 G_n(r)&=\frac1{\sqrt{2\nu_0\nu_n}}
         \int_t^1\varpi(u)P_n^{(1/2,3/2)}(u)\,\dd u
 =\frac{\varpi_+(t)\widehat P_{n-1}(t)}{2n\sqrt{2\nu_0\nu_n}},\qquad n\ge1.
\end{align*}
Indeed,
\[
 (\varpi_+\widehat P_{n-1})'=-2n\varpi P_n^{(1/2,3/2)},\qquad
 (P_n^{(1/2,3/2)})'=\frac{n+3}{2}\widehat P_{n-1}.
\]
If $\mu_j=\int_{-1}^1\varpi_+\widehat P_j^2$ and $\widehat p_j=\widehat P_j/\sqrt{\mu_j}$, integration
by parts yields $\mu_{n-1}=4n\nu_n/(n+3)$.  Thus, for $m,n\ge1$,
\[
 k_{mn}=\frac8{3\sqrt{m(m+3)n(n+3)}}
 \int_{-1}^1(1+t)^2\widehat p_{m-1}(t)\widehat p_{n-1}(t)\varpi_+(t)\,\dd t.
\]
The recurrence
\[
 t\widehat p_j=a_{j+1}\widehat p_{j+1}+b_j\widehat p_j+a_j\widehat p_{j-1},\qquad
 a_j=\frac{\sqrt{j(j+4)}}{2(j+2)},\quad
 b_j=\frac1{(j+2)(j+3)},\quad \widehat p_{-1}=0,
\]
applied twice proves \eqref{4.9}.
\end{proof}

\begin{lemma}\label{4.10}
With $W:=V_{f_0}-\frac73\Sigma$, one has
\begin{equation*}
 -\frac79\le W(r)\le\frac19,\qquad r\ge0.
\end{equation*}
\end{lemma}
\begin{proof}
Direct Newton integration gives
\begin{align*}
 V_{f_0}(3x/2)
 &=\frac{8\arctan x}{3x}
   +\frac{8(3x^2+5)}{9(1+x^2)^2},\\
 W(3x/2)
 &=\frac{8\arctan x}{3x}
   -\frac{7x^4+46x^2+23}{9(1+x^2)^2}.
\end{align*}
The claimed inequalities are equivalent to
\[
 \frac{2x(2x^2+1)}{3(1+x^2)^2}\le\arctan x
 \le\frac{x(x^4+6x^2+3)}{3(1+x^2)^2}.
\]
For the lower bound,
\[
 \frac{\dd}{\dd x}
 \left(\arctan x-\frac{2x(2x^2+1)}{3(1+x^2)^2}\right)
 =\frac{7x^4+1}{3(1+x^2)^3}>0.
\]
If $q(x)=x(x^4+6x^2+3)/(3(1+x^2)^2)-\arctan x$, then
\[
 q'(x)=\frac{x^2(x^2-1)(x^2-3)}{3(1+x^2)^3},\qquad
 q(0)=0,\qquad q(\sqrt3)=\frac{5\sqrt3}{8}-\frac\pi3>0.
\]
The last inequality follows from $\sqrt3>12/7$ and $\pi<22/7$.
The sign of $q'$ proves the upper bound.
\end{proof}

\input{boson_f0_figures_latex/fig_f0_potentials.tex}

\subsection{Quadratic estimates on the two-constraint space}
For a symmetric band matrix $B=(b_{ij})_{i,j\ge N}$, put
\[
 \omega_n:=\frac1{\sqrt{(n+1)(n+2)}},\qquad
 \rho_i(B):=b_{ii}-\sum_{j\ne i}|b_{ij}|\frac{\omega_j}{\omega_i}.
\]
For every finitely supported real sequence,
\begin{equation}\label{4.11}
 \sum_{i,j\ge N}b_{ij}c_ic_j
 =\sum_{i\ge N}\rho_i(B)c_i^2
 +\sum_{N\le i<j}|b_{ij}|\omega_i\omega_j
 \left(\frac{c_i}{\omega_i}
       +\operatorname{sgn}(b_{ij})\frac{c_j}{\omega_j}\right)^2.
\end{equation}
All form comparisons below are first proved on finite Jacobi sums.
The form-core property then extends them to the indicated closed spaces.

\begin{lemma}\label{4.12}
On $Y_2$,
\begin{equation*}
 K_{f_0}\preceq\frac18A,\qquad
 L_{\mathrm{BO}}\succeq\frac{67}{98}A+\frac1{14}\Id,\qquad
 \frac{205}{252}A\preceq H_Z\preceq\frac73A.
\end{equation*}
\end{lemma}
\begin{proof}
For $A/8-K_{f_0}$ restricted to $Y_3$, the first off-diagonal entries
are positive, the second are negative, and
\[
 \rho_3=\frac{17}{90},\qquad
 \rho_4=\frac7{40},\qquad
 \rho_n=\frac18\quad(n\ge5).
\]
Hence $(A/8-K_{f_0})[z]\ge\|z\|_2^2/8$ for $z\in Y_3$.
The remaining entries are
\[
 (A/8-K_{f_0})_{22}=\frac1{60},\quad
 (A/8-K_{f_0})_{23}=-\frac{31\sqrt{15}}{5400},\quad
 (A/8-K_{f_0})_{24}=-\frac{\sqrt{10}}{100}.
\]
For $h=a\phi_2+z$, $z\in Y_3$, the preceding tail bound gives
\begin{align*}
 (A/8-K_{f_0})[h]
 &\ge\frac{a^2}{60}
 -2a\left\langle\frac{31\sqrt{15}}{5400}\phi_3
                   +\frac{\sqrt{10}}{100}\phi_4,z\right\rangle
 +\frac18\|z\|_2^2\\
 &=\frac{229}{48600}a^2
 +\frac18\left\|z-a\left(\frac{31\sqrt{15}}{675}\phi_3
                         +\frac{2\sqrt{10}}{25}\phi_4\right)\right\|_2^2.
\end{align*}
Here
\[
 \frac1{60}-8\left(\frac{961}{1944000}+\frac1{1000}\right)
 =\frac{229}{48600}>0.
\]
This proves $K_{f_0}\preceq A/8$.

For $L_{\mathrm{BO}}-67A/98-\Id/14$ restricted to $Y_2$, its first
off-diagonal entry is
\[
 \frac{40n-67}{147}\sqrt{1-\frac1{(n+2)^2}}>0\qquad(n\ge2),
\]
and the row margins are
\[
 \rho_2=0,\qquad \rho_n=\frac{12}{49}\quad(n\ge3).
\]
Equation \eqref{4.11} proves the second inequality.
Since
\[
 H_Z=\frac{14}{9}L_{\mathrm{BO}}-M_W-2K_{f_0},
\]
we obtain
\[
 H_Z\succeq\frac{14}{9}
 \left(\frac{67}{98}A+\frac1{14}\Id\right)
 -\frac19\Id-\frac14A=\frac{205}{252}A.
\]
The upper bound follows from positivity of both Hartree terms.
\end{proof}

The mixed-potential estimate in \cref{5.0} will use the majorant
$17/10+12s^2$.  This shape is adapted to the basis:
\[
 s=\frac{1+\chi}{2},\qquad
 (M_s)_{mn}=0\ (|m-n|>1),\qquad
 (M_{s^2})_{mn}=0\ (|m-n|>2).
\]
Thus the required bound reduces to a band-form comparison rather than
a sharp optimization of the cubic term.

\begin{lemma}\label{4.13}
For $s(r)=(1+4r^2/9)^{-1}$,
\begin{equation*}
 \frac{17}{10}\|h\|_2^2+12\int_0^\infty s(r)^2h(r)^2\,\dd r
 \le\frac{41}{20}A[h],\qquad h\in Y_2.
\end{equation*}
\end{lemma}
\begin{proof}
Let $\mathcal R=41A/20-17\Id/10-12M_{s^2}$.
We compute $M_{s^2}=M_s^2$ on the full space before taking its restriction
to $Y_2$.  Equations \eqref{4.4} and
\eqref{4.7} give, for $n\ge2$,
\begin{align*}
 \mathcal R_{nn}
 &=\frac{41n^3+60n^2-107n-171}{30(n+1)(n+2)},\\
 \mathcal R_{n,n+1}
 &=\frac{(41n^3+66n^2-269n-384)\sqrt{(n+1)(n+3)}}
         {60(n+1)(n+2)(n+3)},\\
 \mathcal R_{n,n+2}
 &=-\frac34\sqrt{\frac{(n+1)(n+4)}{(n+2)(n+3)}}.
\end{align*}
For $n\ge3$, the first off-diagonal is positive.  On the indices $n\ge3$,
the weighted margins are
\[
 \rho_3(\mathcal R)=\frac{137}{120},\qquad
 \rho_4(\mathcal R)=\frac{59}{40},\qquad
 \rho_n(\mathcal R)=\frac7{20}\quad(n\ge5).
\]
The remaining row is
\[
 \mathcal R_{22}=\frac{61}{120},\qquad
 \mathcal R_{23}=-\frac{11\sqrt{15}}{120},\qquad
 \mathcal R_{24}=-\frac{9\sqrt{10}}{40}.
\]
For $h=a\phi_2+\sum_{n\ge3}c_n\phi_n$, completing squares yields
\begin{align*}
 \mathcal R[h]\ge{}&
 \frac{13267}{242490}a^2
 +\frac{137}{120}\left(c_3-\frac{11\sqrt{15}}{137}a\right)^2\\
 &+\frac{59}{40}\left(c_4-\frac{9\sqrt{10}}{59}a\right)^2
 +\frac7{20}\sum_{n\ge5}c_n^2,
\end{align*}
since
\[
 \frac{61}{120}-\frac{121}{1096}-\frac{81}{236}
 =\frac{13267}{242490}>0.
\]
The bound first holds for finite sums; $X_A$-continuity gives it on $Y_2$.
\end{proof}

\subsection{The centre defect}
\begin{lemma}\label{4.14}
For any $ h\in Y_2$, we have 
\begin{equation}\label{4.15}
 |\langle\varepsilon_0,h\rangle|
 \le\frac1{64}A[h]^{1/2}.
\end{equation}
\end{lemma}
\begin{proof}
Since $f_0\in\mathcal D(A)$ and $V_{f_0}\in L^\infty$, we have
$\varepsilon_0\in L^2$. Here, we aim to estimate the norm of its restriction to
$(Y_2,A)$, rather than its $L^2$ norm.

Write $P_n=P_n^{(1/2,3/2)}$ and
\[
 \varpi(t)=(1-t)^{1/2}(1+t)^{3/2},\qquad
 \kappa_n=\prod_{j=0}^{n-1}\frac{2j+3}{2j+4},\qquad \kappa_0=1.
\]
The Chebyshev polynomials $U_j$ satisfy
$U_j(\cos\theta)=\sin((j+1)\theta)/\sin\theta$ and
\[
 (1+t)P_n(t)=\kappa_n
 \left(U_n(t)+\frac{n+1}{n+2}U_{n+1}(t)\right),\qquad
 \nu_n=\pi\kappa_n^2\frac{n+1}{n+2}.
\]
For completeness, the polynomial in parentheses vanishes at $t=-1$.
After division by $1+t$, it has degree $n$ and is orthogonal to every
polynomial of degree less than $n$ for the weight $\varpi$. Indeed,
this follows from the orthogonality of $U_n,U_{n+1}$ for
$\sqrt{1-t^2}\,\dd t$. Comparing leading coefficients gives the first
identity. The coefficient of $U_n$ in $P_n$ is
$2\kappa_n(n+1)/(n+2)$; multiplying the first identity by
$P_n(t)\sqrt{1-t^2}$ and integrating gives the formula for $\nu_n$.

For every integer $m\ge1$, integration by parts gives
\[
 \int_0^\pi\theta\sin(m\theta)\,\dd\theta
 =\frac{(-1)^{m+1}\pi}{m}.
\]
Therefore, for $j\ge1$,
\begin{align*}
 \int_{-1}^1(1+t)\arccos(t)U_j(t)\,\dd t
 =\int_0^\pi\theta(1+\cos\theta)\sin((j+1)\theta)\,\dd\theta
 =\frac{(-1)^{j+1}\pi}{j(j+1)(j+2)}.
\end{align*}
For $j=0$ the same integral is $3\pi/4$. Using the identity for
$(1+t)P_n$, we obtain
\[
 \int_{-1}^1(1+t)^2\arccos(t)P_n(t)\,\dd t
 =\frac{2(-1)^{n+1}\pi\kappa_n(2n+3)}
 {n(n+1)(n+2)^2(n+3)},\qquad n\ge1,
\]
whereas the integral for $n=0$ is $5\pi/6$.

We now compute $k_{n0}=\langle\phi_n,K_{f_0}\phi_0\rangle$.
Since $\phi_0=\sqrt2 f_0$, we have
$K_{f_0}\phi_0=V_{f_0}\phi_0$. With $t=\chi(r)$, the change of
variables used in \cref{4.2} and the explicit Newton potential give
\begin{align*}
 k_{n0}
 =\frac1{\sqrt{\nu_0\nu_n}}\biggl[
 &\frac43\int_{-1}^1(1+t)^2\arccos(t)P_n(t)\,\dd t\\
 &+\frac49\int_{-1}^1\varpi(t)(t^2+5t+4)P_n(t)\,\dd t\biggr].
\end{align*}
The second integral vanishes for $n\ge3$. For the remaining indices,
\[
 P_0=1,\qquad P_1=2t-\frac12,\qquad
 P_2=\frac{15}{4}t^2-\frac54t-\frac58,
\]
and direct integration gives
\[
 \int_{-1}^1\varpi(t)(t^2+5t+4)P_n(t)\,\dd t
 =\begin{cases}
  11\pi/4,&n=0,\\
  \pi,&n=1,\\
  5\pi/64,&n=2.
 \end{cases}
\]
These values follow by writing $\varpi=(1+t)\sqrt{1-t^2}$ and using
\[
 \int_{-1}^1\sqrt{1-t^2}\,\dd t=\frac\pi2,\quad
 \int_{-1}^1t^2\sqrt{1-t^2}\,\dd t=\frac\pi8,\quad
 \int_{-1}^1t^4\sqrt{1-t^2}\,\dd t=\frac\pi{16},
\]
with all odd moments equal to zero. Substitution yields
\[
 \begin{gathered}
 k_{00}=\frac{14}{3},\qquad k_{10}=\frac7{3\sqrt3},\qquad
 k_{20}=\frac1{15\sqrt6},\\
 k_{n0}=\frac{(-1)^{n+1}8\sqrt2(2n+3)}
 {3n(n+1)^{3/2}(n+2)^{3/2}(n+3)},\qquad n\ge3.
 \end{gathered}
\]
Since $e_n:=\langle\varepsilon_0,\phi_n\rangle
=(k_{n0}-\beta_0A_{n0})/\sqrt2$, combine with \eqref{4.4}, we have 
\[
 \begin{gathered}
 e_0=e_1=0,\qquad e_2=\frac{\sqrt3}{90},\qquad
 e_3=\frac{\sqrt5}{150},\\
 e_n^2=\frac{64(2n+3)^2}
 {9n^2(n+1)^3(n+2)^3(n+3)^2},\qquad n\ge3.
 \end{gathered}
\]

To use the form norm, take $u=(\sqrt3/300)\phi_2\in Y_2$.
This coefficient is $e_2/A_{22}$, since $A_{22}=10/3$. Hence
\[
 A[u]=\langle\varepsilon_0,u\rangle=\frac1{9000}.
\]
The tridiagonal matrix of $A$ also gives
\[
 \langle\varepsilon_0-Au,\phi_2\rangle=0,\qquad
 \langle\varepsilon_0-Au,\phi_3\rangle=\frac{\sqrt5}{300},\qquad
 \langle\varepsilon_0-Au,\phi_n\rangle=e_n\quad(n\ge4).
\]
Let $\Pi_2$ denote the $L^2$ projection onto
$\{\phi_0,\phi_1\}^{\perp}$. The variational identity for the squared
dual norm is
\[
 \sup_{0\ne h\in Y_2}\frac{|\langle\varepsilon_0,h\rangle|^2}{A[h]}
 =\sup_{h\in Y_2}\{2\langle\varepsilon_0,h\rangle-A[h]\}.
\]
It follows by optimizing the scalar multiple of any nonzero $h$.
For $h=u+v$, $v\in Y_2$, completing the square and using
$A[v]\ge\|v\|_2^2$ yield
\begin{align*}
 2\langle\varepsilon_0,h\rangle-A[h]
 &=2\langle\varepsilon_0,u\rangle-A[u]
   +2\langle\Pi_2(\varepsilon_0-Au),v\rangle-A[v]\\
 &\le\frac1{9000}+\|\Pi_2(\varepsilon_0-Au)\|_2^2\\
 &=\frac1{6000}+\sum_{n\ge4}e_n^2.
\end{align*}
Here the projection is necessary: $Au$ can have a $\phi_1$ component,
which does not pair with $v\in Y_2$.

Finally,
\[
 e_4^2=\frac{121}{2976750}<\frac1{24000},\qquad
 e_5^2=\frac{169}{16669800}<\frac1{90000}.
\]
For $n\ge6$, the inequalities
$2n+3\le\frac52(n+1)$, $n(n+3)\ge(n+1)^2$, and $n+2\ge n+1$ give
\[
 e_n^2\le\frac{400}{9(n+1)^8},\qquad
 \sum_{n\ge6}e_n^2
 <\frac{400}{9}\int_6^\infty t^{-8}\,\dd t
 =\frac{25}{1102248}<\frac1{44000}.
\]
Thus
\[
 \sup_{0\ne h\in Y_2}\frac{|\langle\varepsilon_0,h\rangle|^2}{A[h]}
 <\frac1{6000}+\frac1{24000}+\frac1{90000}+\frac1{44000}
 <\frac1{4096}.
\]
Taking square roots proves \eqref{4.15}.
\end{proof}

\section{Optimal dilation and a global potential bound}
\label{5.0}

A maximizing direction need not initially be close to $f_0$.
We first select its scale by a global overlap maximum.  Positivity then
converts the resulting inequalities into bounds on the mixed density
$f_0g$, without any assumption on $A[g-f_0]$.

\begin{lemma}\label{5.1}
Let $f\ge0$ be a maximizer of $\mathfrak b$ on $A[f]=1$.
There are $a,\lambda>0$ such that
\begin{equation*}
 g=af_\lambda=f_0+w,\qquad w\in Y_2,
 \qquad\mathcal B(g)=\beta_*.
\end{equation*}
Moreover,
\begin{equation}\label{5.2}
 \langle g,f_0^{(q)}\rangle\le\frac12\quad(q>0),\qquad
 \langle g,f_0\rangle=\frac12,\qquad
 f_0^{(q)}(r):=q^{-1/2}f_0(r/q).
\end{equation}
\end{lemma}
\begin{proof}
Set $S(\lambda):=\langle f_\lambda,f_0\rangle$ for $\lambda>0$.
Since $f\ge0$ is nonzero and $f_0>0$ on $\mathbb R_+$,
Cauchy--Schwarz gives $0<S(\lambda)<\infty$.
We first show that $S$ attains its maximum at a finite positive scale.
The map
\[
 (\mathcal Uh)(t):=e^{t/2}h(e^t)
\]
is unitary from $L^2(\mathbb R_+)$ onto $L^2(\mathbb R)$, and
\[
 \mathcal U(f_\lambda)(t)=(\mathcal Uf)(t+\log\lambda).
\]
Writing $F=\mathcal Uf$ and $G=\mathcal Uf_0$, we have
\[
 S(e^s)=\langle F(\cdot+s),G\rangle_{L^2(\mathbb R)}.
\]
Strong continuity of translations, obtained first for smooth compactly
supported functions and then by $L^2$ approximation, proves continuity.
For completeness, let $F_R=\mathbf1_{[-R,R]}F$ and
$G_R=\mathbf1_{[-R,R]}G$.  If $|s|>2R$, their shifted supports are
disjoint, so
\[
 |S(e^s)|\le
 \|F-F_R\|_2\|G\|_2+\|F_R\|_2\|G-G_R\|_2.
\]
The right-hand side tends to zero as $R\to\infty$.  Thus $S(\lambda)$
tends to zero as $\lambda\downarrow0$ and as $\lambda\to\infty$.
Since $S(1)>0$, its maximum is attained on a compact subinterval of
$(0,\infty)$.  Choose one maximizing scale $\lambda$ and set
\[
 a:=\frac1{2S(\lambda)},\qquad g:=af_\lambda.
\]
The dilation identity $A[f_\lambda]=M(f)+\lambda T(f)$ shows that
$g\in X_A$, and $g\ge0$.

For every $q>0$, the change of variables $r=qs$ gives
\[
 \langle f_\lambda,f_0^{(q)}\rangle
 =\int_0^\infty (q\lambda)^{1/2}f(q\lambda s)f_0(s)\,\dd s
 =S(q\lambda).
\]
Consequently,
\[
 \langle g,f_0^{(q)}\rangle
 =\frac{S(q\lambda)}{2S(\lambda)}\le\frac12,
 \qquad \langle g,f_0\rangle=\frac12,
\]
which proves \eqref{5.2}.

We differentiate this last scalar pairing only through the test
function.  The explicit formula for $f_0$ gives
\[
 \Lambda f_0(r)
 =\left(\frac32-\frac{4r^2}{r^2+9/4}\right)f_0(r),\qquad
 \partial_q f_0^{(q)}=-q^{-1}(\Lambda f_0)^{(q)}.
\]
For $1/2\le q\le2$, the absolute value of this derivative is bounded
by $Cr/(1+r^2)^2$, an $L^2(\mathbb R_+)$ function.  Hence
$q\mapsto f_0^{(q)}$ is differentiable in $L^2$ at $q=1$, with
derivative $-\Lambda f_0$.  Since the pairing has its maximum there,
\[
 0=\left.\frac{\dd}{\dd q}\right|_{q=1}
        \langle g,f_0^{(q)}\rangle
   =-\langle g,\Lambda f_0\rangle.
\]
No differentiability of the dilation orbit of $f$ is required.

Since $f_0=\phi_0/\sqrt2$, \eqref{4.6} and \eqref{4.4} imply
\[
 \Lambda f_0=\sqrt{\frac38}\,\phi_1,
 \qquad Df_0=f_0+\frac1{\sqrt6}\phi_1.
\]
It follows that $\langle g,\phi_1\rangle=0$.  With $w=g-f_0$, the
normalization of $g$ and the orthogonality of $\phi_0,\phi_1$ give
\[
 \langle w,f_0\rangle=0,\qquad
 \langle w,Df_0\rangle
 =\langle w,f_0\rangle+\frac1{\sqrt6}\langle w,\phi_1\rangle=0.
\]
Thus $w\in Y_2$.  Finally, \eqref{3.3} gives
$\mathcal B(f)=\mathfrak b(f)=\beta_*$, and the invariance in
\eqref{3.2} yields $\mathcal B(g)=\beta_*$.
\end{proof}

\begin{lemma}\label{5.3}
Let $g\in X_A$, $g\ge0$, satisfy \eqref{5.2}.  Then
\begin{equation}\label{5.4}
 V[f_0g](r)\le\frac{17}{10}
        +\frac{12}{(1+4r^2/9)^2},\qquad r\ge0.
\end{equation}
Consequently,
\begin{equation}\label{5.5}
 \mathfrak C(f_0g,h^2)\le\frac{41}{20}A[h],\qquad h\in Y_2.
\end{equation}
\end{lemma}
\begin{proof}
Put $x=2r/3$ and define
\[
 \dd\mu(x)=3f_0(3x/2)g(3x/2)\,\dd x.
\]
The assumptions $g\ge0$ and $\langle g,f_0\rangle=1/2$ imply
$\mu\ge0$ and $\int\dd\mu=1$. The explicit formula for $f_0$ gives
\[
 \frac{f_0^{(q)}(3x/2)}{f_0(3x/2)}
 =q^{5/2}\frac{(1+x^2)^2}{(x^2+q^2)^2}.
\]
Thus \eqref{5.2} becomes
\begin{equation}\label{5.6}
 \int_0^\infty q^{5/2}\frac{(1+x^2)^2}{(x^2+q^2)^2}\,\dd\mu(x)
 \le1\quad(q>0),\qquad
 \int_0^\infty\frac{1-x^2}{1+x^2}\,\dd\mu(x)=\frac14.
\end{equation}
For the second identity, the integral in the first inequality equals
$1$ at $q=1$, so its derivative vanishes there. Indeed,
\[
 \left.q\partial_q\left(q^{5/2}\frac{(1+x^2)^2}{(x^2+q^2)^2}\right)
 \right|_{q=1}
 =\frac52-\frac4{1+x^2}
 =\frac12-2\frac{1-x^2}{1+x^2}.
\]
Differentiation under the integral is valid because the integrand and
its first $q$-derivative are uniformly bounded in $x\ge0$ for
$1/2\le q\le2$.

We first bound the potential at the origin. For $a,b>0$, partial
fractions give
\[
 \int_0^\infty\frac{q^2}{(q^2+a^2)(q^2+b^2)}\,\dd q
 =\frac\pi{2(a+b)}.
\]
Differentiating once in each parameter and dividing by $4ab$ gives
\[
 \int_0^\infty\frac{q^2\,\dd q}{(q^2+x^2)^2(q^2+1)^2}
 =\frac\pi{4x(1+x)^3},\qquad x>0.
\]
The differentiations are dominated when $a,b$ range over compact
subsets of $(0,\infty)$. Also, with $y=q^2$,
\[
 \int_0^\infty\frac{q^{-1/2}}{(1+q^2)^2}\,\dd q
 =\frac12\frac{\Gamma(1/4)\Gamma(7/4)}{\Gamma(2)}
 =\frac{3\pi\sqrt2}{8}.
\]
Multiplying the first inequality in \eqref{5.6} by
$q^{-1/2}(1+q^2)^{-2}$ and using Tonelli's theorem therefore gives
\begin{equation}\label{5.7}
 \int_0^\infty\frac{(1+x^2)^2}{x(1+x)^3}\,\dd\mu(x)
 \le\frac{3\sqrt2}{2}.
\end{equation}
We shall combine this with the elementary inequality
\begin{equation}\label{5.8}
 \frac1x\le\frac{11}{20}
 +\frac{24}{25}\frac{1-x^2}{1+x^2}
 +\frac{57}{50}\frac{(1+x^2)^2}{x(1+x)^3},\qquad x>0.
\end{equation}
To verify it, multiply the difference of the two sides by the positive
quantity $100x(1+x^2)(1+x)^3$. The result is
\[
 F(x)=73x^6-223x^5+70x^4+12x^3+395x^2-149x+14.
\]
For $z=(1,x,x^2,x^3)^T$, direct multiplication gives
\[
 10F(x)=z^T\mathbf Gz,\qquad
 \mathbf G=\begin{pmatrix}
 140&-745&-54&207\\
 -745&4058&-147&-926\\
 -54&-147&2552&-1115\\
 207&-926&-1115&730
 \end{pmatrix}.
\]
Its leading principal minors are
\[
 140,\qquad13095,\qquad6732432,\qquad2435840.
\]
All are positive, so $\mathbf G$ is positive definite and
\eqref{5.8} follows. Integrating and using \eqref{5.6}--\eqref{5.7},
we obtain
\[
 \int_0^\infty\frac1x\,\dd\mu(x)
 \le\frac{11}{20}+\frac{24}{25}\frac14
       +\frac{57}{50}\frac{3\sqrt2}{2}
 =\frac{79+171\sqrt2}{100}.
\]
In particular the integral is finite, and
\begin{equation*}
 V[f_0g](0)
 =\frac{4\pi}{3}\int_0^\infty\frac1x\,\dd\mu(x)\le\frac{\pi(79+171\sqrt2)}{75}<\frac{269}{20}.
\end{equation*}

We next obtain bounds away from the origin. Write
\[
 \mathcal V(x):=V[f_0g](3x/2),\qquad
 p(x):=\frac{17}{10}+\frac{12}{(1+x^2)^2},\qquad
 F_q(t):=\frac{(t^2+q^2)^2}{t(1+t^2)^2}.
\]
For any $q>0$, positivity and \eqref{5.6} give
\begin{align*}
 \mathcal V(x)
 =\frac{4\pi}{3}\int_0^\infty\frac{\dd\mu(t)}{\max\{x,t\}}\le\frac{4\pi}{3}q^{-5/2}
 \sup_{t>0}\frac{(t^2+q^2)^2}{(1+t^2)^2\max\{x,t\}}.
\end{align*}
Indeed, multiply the integrand from \eqref{5.6} by the ratio displayed
in the supremum. For $0<q\le1$, the ratio
$(t^2+q^2)/(1+t^2)$ is nondecreasing, since its derivative is
$2t(1-q^2)/(1+t^2)^2$. The supremum over $0<t\le x$ is therefore
attained at $t=x$, and hence
\[
 \mathcal V(x)\le\frac{4\pi}{3}q^{-5/2}\sup_{t\ge x}F_q(t),
 \qquad 0<q\le1.
\]
We compare this estimate with $p$ on four intervals.

(I) For $0<x\le1/10$, the Newton kernel shows that $\mathcal V$ is
nonincreasing. Since $p$ is also nonincreasing,
\[
 \mathcal V(x)<\frac{269}{20}<p(1/10)\le p(x).
\]

(II) For $1/10\le x\le1/5$, choose $q=4/9$. Logarithmic differentiation gives
\[
 \frac{F_q'(t)}{F_q(t)}
 =\frac{-t^4+(3-5q^2)t^2-q^2}{t(t^2+q^2)(1+t^2)}.
\]
For $t\le1/5$, its numerator is bounded above by
\[
 \frac{163}{81\cdot25}-\frac{16}{81}=-\frac{79}{675}<0.
\]
Thus $F_{4/9}$ decreases on $[x,1/5]$. For $t\ge1/5$, we have
\[
 F_{4/9}(t)\le\frac{(t^2+1/5)^2}{t(1+t^2)^2}\le\frac25.
\]
To check the last inequality, the logarithmic derivative of the middle
expression has numerator $-t^4+2t^2-1/5$. Its positive critical points
are $t_\pm=\sqrt{1\pm2/\sqrt5}$; $t_-$ is a local minimum and $t_+$
is a local maximum. The square of its value at $t_+$ is
$1/10+\sqrt5/50<4/25$, its value at $1/5$ is $45/169<2/5$, and its
limit at infinity is zero. Since
\[
 \frac{4\pi}{3}(4/9)^{-5/2}=\frac{81\pi}{8}<\frac{891}{28},
\]
we obtain
\[
 \mathcal V(x)\le
 \max\left\{\frac{891}{28}
 \frac{(x^2+16/81)^2}{x(1+x^2)^2},\frac{891}{70}\right\}.
\]
The constant is less than $p(x)$ because
$p(1/5)-891/70=391/5915>0$. For the other term, put
\[
 J_1(x):=x\left[\frac{17}{10}(1+x^2)^2+12\right]
        -\frac{891}{28}(x^2+16/81)^2.
\]
Then
\[
 J_1(1/10)>0
\]
and, throughout this interval,
\begin{align*}
 J_1'(x)
 &=\frac{595x^4-8910x^3+714x^2-1760x+959}{70}\\
 &\ge\frac{959-8910/125-1760/5}{70}>0.
\end{align*}
Thus $J_1>0$, as required.

(III) For $1/5\le x\le3/5$, choose $q=3/5$. The numerator of
$F_q'/F_q$ is $-(t^2-3/5)^2$, so $F_q$ is nonincreasing and
\[
 \mathcal V(x)\le\frac{100\pi\sqrt{15}}{81}
 \frac{(x^2+9/25)^2}{x(1+x^2)^2}
 <\frac{76}{5}\frac{(x^2+9/25)^2}{x(1+x^2)^2}.
\]
Define
\[
 J_2(x):=x\left[\frac{17}{10}(1+x^2)^2+12\right]
        -\frac{76}{5}(x^2+9/25)^2.
\]
Its endpoint values are positive, i.e. $J_2(1/5)>0$, 
 $J_2(3/5)>0.$
Moreover,
\begin{align*}
 J_2''(x)
 &=34x^3-\frac{912}{5}x^2+\frac{102}{5}x-\frac{2736}{125}\\
 &\le34(3/5)^3+\frac{102}{5}\frac35-\frac{2736}{125}
 =-\frac{288}{125}<0.
\end{align*}
Concavity places $J_2$ above the chord joining its positive endpoint
values. This proves $\mathcal V\le p$ on the interval.

(IV) For $x\ge3/5$, the probability normalization gives
\[
 \mathcal V(x)\le\frac{4\pi}{3x}<\frac{88}{21x}.
\]
Set $\zeta(x):=xp(x)$. Direct differentiation gives
\[
 \zeta'(x)=\frac{17}{10}+\frac{12(1-3x^2)}{(1+x^2)^3},\qquad
 \zeta''(x)=\frac{144x(x^2-1)}{(1+x^2)^4}.
\]
The values
\[
 \zeta(3/5)=\frac{70989}{14450},\quad
 \zeta(1)=\frac{47}{10},\quad
 \zeta(3/2)=\frac{14379}{3380},\quad
 \zeta(2)=\frac{109}{25}
\]
all exceed $21/5$. On $[3/5,1]$, concavity gives the same lower bound.
On $[1,3/2]$, the derivative is increasing but remains negative because
$\zeta'(3/2)<0$, so $\zeta(x)\ge\zeta(3/2)>21/5$.
On $[2,\infty)$ it is positive because
$\zeta'(2)>0$ and $\zeta''>0$.

It remains to treat $[3/2,2]$. On this interval,
\[
 \zeta''(x)\ge\frac{144(3/2)(5/4)}{5^4}
 =\frac{54}{125}>\frac25,
\]
while
\[
 \zeta(5/3)=\frac{7343}{1734}>\frac{21}{5},\qquad
 |\zeta'(5/3)|=\frac{3331}{49130}<\frac7{100}.
\]
Taylor's formula, with $y=x-5/3$, gives
\[
 \zeta(x)\ge\zeta(5/3)+\zeta'(5/3)y+\frac15y^2
 \ge\zeta(5/3)-\frac54|\zeta'(5/3)|^2
 >\frac{88}{21}.
\]
Thus
$\mathcal V(x)\le p(x)$ for all $x>0$.

Since $\int t^{-1}\,\dd\mu(t)<\infty$ and
$1/\max\{x,t\}\le1/t$, dominated convergence gives continuity at
$x=0$, so \eqref{5.4} holds there as well. Multiplying \eqref{5.4}
by $h(r)^2$ and integrating, we conclude from \cref{4.13} that
\[
 \mathfrak C(f_0g,h^2)
 \le\frac{17}{10}\|h\|_2^2
      +12\int_0^\infty\frac{h(r)^2}{(1+4r^2/9)^2}\,\dd r
 \le\frac{41}{20}A[h],\qquad h\in Y_2.
\]
This proves \eqref{5.5}.
\end{proof}

\section{Localization, strict concavity, and uniqueness}
\label{6.0}

In this section, we prove \cref{1.2} by first localizing all maximizers and then
establishing strict concavity. By \cref{5.1}, every nonnegative maximizer
can be written as $g=f_0+w$, $w\in Y_2$, after an amplitude change and a
dilation. We work with $\widetilde\Psi(w)=\mathcal B(f_0+w)$ from
\cref{3.0}.
We shall also use a cubic estimate valid on all of $Y_2$.
Cauchy--Schwarz for $\mathfrak C$, \cref{4.12}, and \eqref{3.6} give
\begin{equation}\label{6.1}
 \bigl|\mathfrak C(f_0h,h^2)\bigr|^2
 \le K_{f_0}[h]\mathfrak b(h)
 \le\frac58A[h]^3<\frac{16}{25}A[h]^3
 \qquad(h\ne0,\ h\in Y_2).
\end{equation}

\subsection{Global localization}
\begin{lemma}\label{6.2}
Let $g=f_0+w$ be a representative furnished by \cref{5.1},
and put
\[
 \rho=A[w]^{1/2},\quad \ell=\langle\varepsilon_0,w\rangle,\quad
 \delta=\beta_*-\beta_0\ge0,\quad
 H=H_Z[w],\quad c=\mathfrak C(f_0w,w^2).
\]
Then
\begin{equation}\label{6.3}
 c=H+\delta(1+\rho^2)-3\ell,
\end{equation}
and
\begin{equation}\label{6.4}
 4\beta_*M(w)T(w)-\mathfrak b(w)
 =2H+\delta(3+2\rho^2)-8\ell.
\end{equation}
\end{lemma}
\begin{proof}
Since $g=f_0+w$ is a maximizer of $\mathcal B$ on $X_A\setminus\{0\}$,
differentiating the quotient gives the equation
\[
 V_g g=2\beta_*\bigl(T(g)g+M(g)Dg\bigr)
 \qquad\text{in }X_A^*.
\]
Here the amplitude and dilation chosen in \cref{5.1} need not preserve
$A[g]=1$, so the equation is not in general $V_g g=\beta_*Ag$.
The conditions $w\in Y_2$ imply
\[
 \langle g,f_0\rangle=\langle Dg,f_0\rangle=\frac12,
 \qquad A[g]=1+\rho^2.
\]
Pairing the Euler equation with $f_0$ therefore yields
\[
 \mathfrak C(g^2,gf_0)=\beta_*(1+\rho^2).
\]
Using $A[f_0,w]=0$ and the definitions of $\ell,H,c$, we have
\[
 \mathfrak C(f_0^2,f_0w)=\ell,\qquad
 \mathfrak C(f_0^2,w^2)+2\mathfrak C(f_0w,f_0w)
 =\beta_0\rho^2-H.
\]
Bilinearity and symmetry now give
\begin{align*}
 \mathfrak C(g^2,gf_0)
 &=\mathfrak C(f_0^2+2f_0w+w^2,f_0^2+f_0w)\\
 &=\beta_0+3\ell+\beta_0\rho^2-H+c.
\end{align*}
Comparison proves \eqref{6.3}.  Testing with the fixed function $f_0$
leaves at least one test factor in each Hartree term, so no pure
quartic term in $w$ occurs in this identity.

The value $\mathcal B(g)=\beta_*$ and \eqref{3.10} also give
\[
 \beta_0+4\ell+2\beta_0\rho^2-2H+4c+\mathfrak b(w)
 =\beta_*\bigl(1+2\rho^2+4M(w)T(w)\bigr).
\]
Consequently,
\begin{align*}
 4\beta_*M(w)T(w)-\mathfrak b(w)
 &=-\delta+4\ell-2H-2\delta\rho^2+4c\\
 &=2H+\delta(3+2\rho^2)-8\ell,
\end{align*}
where the last step uses \eqref{6.3}.  This is \eqref{6.4}.
\end{proof}

We keep the sign of the cubic term instead of estimating it separately:
\[
 c-H=\mathfrak C(f_0g,w^2)-\beta_0A[w]+2K_{f_0}[w].
\]
The common term $\mathfrak C(f_0^2,w^2)$ cancels.  The global
mixed-potential bound therefore controls exactly the combination in
\eqref{6.3}.

\begin{lemma}\label{6.5}
Every representative in \cref{5.1} satisfies
\begin{equation}\label{6.6}
 A[w]^{1/2}<\frac1{20}.
\end{equation}
\end{lemma}
\begin{proof}
Assume $\rho>0$ and use the notation of \cref{6.2}.
By \eqref{5.5},
\[
 c=\mathfrak C(f_0g,w^2)-\mathfrak C(f_0^2,w^2)
 \le\frac{41}{20}\rho^2-\mathfrak C(f_0^2,w^2).
\]
Consequently,
\[
 c-H\le\left(\frac{41}{20}-\frac73\right)\rho^2+2K_{f_0}[w]
 \le-\frac1{30}\rho^2.
\]
By \eqref{6.3} and
\eqref{4.15},
\[
 \delta(1+\rho^2)+\frac1{30}\rho^2
 \le3\ell\le\frac3{64}\rho,
\]
and therefore
\begin{equation}\label{6.7}
 \rho\le\frac{45}{32}<\frac85.
\end{equation}

First suppose $1/20\le\rho\le1/8$.  Since
$\mathfrak b(w)\le4\beta_*M(w)T(w)$, the value identity in the preceding
proof yields
\[
 \delta(1+2\rho^2)+2H\le4\ell+4c.
\]
Using $H\ge4\rho^2/5$, $|\ell|\le\rho/64$, and
\eqref{6.1}, we find
\[
 \frac45\rho\le\frac1{32}+\frac85\rho^2.
\]
The concave polynomial
\[
 P(\rho)=\frac45\rho-\frac85\rho^2-\frac1{32}
\]
satisfies
\[
 P(1/20)=\frac{19}{4000}>0,\qquad
 P(1/8)=\frac7{160}>0.
\]
It is positive throughout this interval, a contradiction.

Next suppose $1/8\le\rho\le8/5$.  Equation
\eqref{6.3} gives
\[
 c\ge\frac45\rho^2-\frac3{64}\rho>0.
\]
Since $c^2\le K_{f_0}[w]\mathfrak b(w)$ and
$K_{f_0}[w]\le\rho^2/8$, we have $\mathfrak b(w)\ge8c^2/\rho^2$.
Substituting in \eqref{6.4}, and using
$4M(w)T(w)\le\rho^4$, yields
\begin{align*}
 \frac73\rho^4\ge{}&\frac85\rho^2-\frac18\rho
 +8\left[\frac45\rho-\frac3{64}
              +\delta\left(\rho+\rho^{-1}\right)\right]^2\\
 &+\delta(1+\rho^2)(3-\rho^2).
\end{align*}
Here $\rho<\sqrt3$, so every contribution involving $\delta$ can be
omitted.  Hence
\[
 \frac73\rho^4\ge\frac{168}{25}\rho^2
                    -\frac{29}{40}\rho+\frac9{512}.
\]
On the other hand, $G(\rho)=7\rho^2/3+29/(40\rho)$ is convex and
\[
 G(1/8)=\frac{5603}{960},\qquad
 G(8/5)=\frac{30847}{4800}<\frac{168}{25}.
\]
Thus $G(\rho)<168/25$ on $[1/8,8/5]$, contradicting the preceding
inequality divided by $\rho^2$.  Combining both exclusions with
\eqref{6.7} proves the lemma.
\end{proof}

\subsection{Uniform strict concavity}
\begin{lemma}\label{6.8}
For $w,v\in Y_2$ with $A[w]^{1/2}\le1/20$,
\begin{equation}\label{6.9}
 \mathrm D^2\widetilde\Psi(w)[v,v]\le-\frac74A[v].
\end{equation}
\end{lemma}

\begin{proof}
Put
\[
 d(w):=\bigl(1+2A[w]+4M(w)T(w)\bigr)^{-1},
\]
\[
 q_3(w):=\mathfrak C(f_0w,w^2),\qquad
 q_4(w):=\mathfrak b(w)-4\beta_0M(w)T(w),
\]
and
\[
 p_0(w):=4\langle\varepsilon_0,w\rangle-2H_Z[w]+4q_3(w)+q_4(w).
\]
Then $\widetilde\Psi=\beta_0+dp_0$.  By
\eqref{3.6} and \eqref{6.1},
\[
 |q_3(w)|\le\frac45A[w]^{3/2},\qquad
 |q_4(w)|\le5A[w]^2,\qquad
 0\le4M(w)T(w)\le A[w]^2.
\]
Since $b(w)\ge0$ and
$0\le4\beta_0M(w)T(w)\le\beta_0A[w]^2$, we have
\[
-\beta_0A[w]^2\le q_4(w)\le5A[w]^2.
\]
As $\beta_0=7/3<5$, this gives $|q_4(w)|\le5A[w]^2$.
Banach's theorem for symmetric multilinear forms on real Hilbert spaces
\cite{Banach1938} identifies their multilinear norms with their diagonal
polynomial norms.  Apply it on $(Y_2,A)$, or first on the finite-dimensional
span of the arguments and then pass by continuity. In particular, if $P$ is a continuous $m$-homogeneous polynomial
on $(Y_2,A)$ and $|P(h)|\le C A[h]^{m/2}$, then
\[
|D^jP(w)[v,\ldots,v]|
\le \frac{m!}{(m-j)!}\,
C A[w]^{(m-j)/2}A[v]^{j/2},
\qquad j=1,2.
\]

Assume $A[v]=1$ and write $\rho:=A[w]^{1/2}\le1/20$.
The homogeneous polynomial bounds imply
\begin{align*}
 |p_0(w)|
 &\le\frac\rho{16}+\frac{14}{3}\rho^2
           +\frac{16}{5}\rho^3+5\rho^4<\frac1{64},\\
 |\mathrm Dp_0(w)[v]|
 &\le\frac1{16}+\frac{28}{3}\rho
           +\frac{48}{5}\rho^2+20\rho^3<\frac9{16},\\
 \mathrm D^2p_0(w)[v,v]
 &\le-4H_Z[v]+\frac{96}{5}\rho+60\rho^2
 \le-4H_Z[v]+\frac{111}{100}.
\end{align*}

For the denominator $d(w)^{-1}=1+2A[w]+4M(w)T(w)$,
\begin{align*}
 \mathrm D(d^{-1})(w)[v]
 &=4A[w,v]+8T(w)\langle w,v\rangle
              +8M(w)\langle D^{1/2}w,D^{1/2}v\rangle,\\
 \mathrm D^2(d^{-1})(w)[v,v]
 &=4A[v]+8T(w)M(v)+8M(w)T(v)\\
 &\quad+32\langle w,v\rangle\langle D^{1/2}w,D^{1/2}v\rangle.
\end{align*}
Since $0\le4MT\le A^2$, polarization on $(Y_2,A)$ gives
\[
 |\mathrm D(d^{-1})(w)[v]|\le4\rho+4\rho^3,\qquad
 |\mathrm D^2(d^{-1})(w)[v,v]|\le4+12\rho^2.
\]
The reciprocal rule is
\begin{align*}
 \mathrm Dd(w)[v]&=-d(w)^2\mathrm D(d^{-1})(w)[v],\\
 \mathrm D^2d(w)[v,v]
 &=2d(w)^3\bigl(\mathrm D(d^{-1})(w)[v]\bigr)^2
   -d(w)^2\mathrm D^2(d^{-1})(w)[v,v].
\end{align*}
Furthermore,
\begin{align*}
 (1+\rho^2)^{-2}&\le d(w)\le1,\\
 |\mathrm Dd(w)[v]|&\le4\rho(1+\rho^2)\le\frac{81}{400},\\
 |\mathrm D^2d(w)[v,v]|&\le4+12\rho^2+32\rho^2(1+\rho^2)^2\le\frac{24}{5}.
\end{align*}
Since
\[
 4d(w)H_Z[v]\ge\frac{4(205/252)}{(1+\rho^2)^2}>\frac{16}{5},
\]
the product rule gives
\begin{align*}
 \mathrm D^2\widetilde\Psi(w)[v,v]
 &=d(w)\mathrm D^2p_0(w)[v,v]
    +2\mathrm Dd(w)[v]\mathrm Dp_0(w)[v]
    +p_0(w)\mathrm D^2d(w)[v,v]\\
 &\le-\frac{16}{5}+\frac{111}{100}
      +2\frac{81}{400}\frac9{16}+\frac1{64}\frac{24}{5}\\
 &=-\frac{5719}{3200}<-\frac74.
\end{align*}
Homogeneity in $v$ proves \eqref{6.9}.
\end{proof}

\subsection{Uniqueness and return to the equation}
\label{6.10}
\begin{proposition}\label{6.11}
The problem $\sup_{A[f]=1}\mathfrak b(f)$ has a unique nonnegative
maximizer $f_*$.  Its real maximizing set is $\{f_*,-f_*\}$.
\end{proposition}
\begin{proof}
Existence and the equality cases of rearrangement were recalled in
\cref{2.0}.  Let $f_1,f_2\ge0$ be two maximizers and choose
$g_j=f_0+w_j$ as in \cref{5.1}.  Then
\[
 A[w_j]^{1/2}<\frac1{20},\qquad
 \mathrm D\widetilde\Psi(w_j)=0\quad(j=1,2).
\]
For $v=w_1-w_2$, the segment $w_2+tv$ lies in the same convex ball, and
\begin{align*}
 0&=(\mathrm D\widetilde\Psi(w_1)-\mathrm D\widetilde\Psi(w_2))[v]\\
 &=\int_0^1\mathrm D^2\widetilde\Psi(w_2+tv)[v,v]\,\dd t
 \le-\frac74A[v].
\end{align*}
Thus $w_1=w_2$.  Consequently $f_1=a(f_2)_\lambda$ for some
$a,\lambda>0$.  The balances \eqref{3.3} give
\[
 \frac12=M(f_1)=\frac{a^2}{2},\qquad
 \frac12=T(f_1)=\frac{a^2\lambda}{2},
\]
so $a=\lambda=1$.  Every real maximizer has constant sign by the equality
case of the fractional diamagnetic inequality; hence the real maximizing
set consists of the asserted pair.
\end{proof}

Set
\begin{equation}\label{6.12}
 \mathcal M_*:=\{f\in X_A:A[f]=1,\ \mathfrak b(f)=\beta_*\}
 =\{f_*,-f_*\},\qquad f_*\ge0.
\end{equation}
By \cref{2.3},
\begin{equation*}
 \beta_*Af_*=V_{f_*}f_*.
\end{equation*}

The Hardy-Kato bound gives, for $f\in X_A$,
\[
 0\le V_f(r)\le V_f(0)
 =4\pi\int_0^\infty\frac{f(s)^2}{s}\,\dd s
 \le2\pi^2T(f).
\]
Thus $V_{f_*}f_*\in L^2$.  The sine transform of the weak Euler equation gives
\[
 \|(1+k)\mathcal F_s f_*\|_2
 =\beta_*^{-1}\|V_{f_*}f_*\|_2<\infty,
 \qquad f_*\in\mathcal D(A).
\]
On $L^2(\mathbb R_+)$,
\begin{align*}
 A^{-1}&=\int_0^\infty e^{-t}e^{-tD}\,\dd t,\\
 e^{-tD}&=\frac{t}{2\sqrt\pi}\int_0^\infty
 u^{-3/2}e^{-t^2/(4u)}e^{-uD^2}\,\dd u,\\
 (e^{-uD^2})(r,s)
 &=\frac1{\sqrt{4\pi u}}
   \left(e^{-(r-s)^2/(4u)}-e^{-(r+s)^2/(4u)}\right)>0
 \quad(r,s,u>0).
\end{align*}
Consequently,
\[
 f_* =\beta_*^{-1}A^{-1}(V_{f_*}f_*)>0
 \quad\text{almost everywhere on }\mathbb R_+.
\]

Define
\[
 F_*:=\beta_*^{-1/2}f_*,\qquad Q(x):=\frac{F_*(|x|)}{|x|}.
\]
Then $Q$ is the unique positive radial ground state and
\begin{equation*}
 \mathcal K(Q)=\mathcal H(Q)=\frac{4\pi}{\beta_*},\qquad
 \mathcal S(Q)=\frac\pi{\beta_*}.
\end{equation*}
Moreover,
\begin{equation}\label{6.13}
 \sup_{u\ne0}\frac{\mathcal H(u)}{\mathcal K(u)^2}
 =\frac{\beta_*}{4\pi}.
\end{equation}
For an arbitrary complex ground state, the equality cases in
\cite[Section~2.1]{FrankLenzmann2010} give
$u=e^{i\vartheta}|u|^*(\,\cdot-x_0)$, where $|u|^*$ is a positive radial
ground state.  Thus $u=e^{i\vartheta}Q(\,\cdot-x_0)$, proving
\cref{1.2}.

\section{Nondegeneracy of the ground state}\label{7.0}
The regularity and decay results of
\cite[Theorem~1.1, Lemmas~2.1-2.2 and 3.3]{FrankLenzmann2010} give
\[
 Q\in\bigcap_{s>0}H^s(\mathbb R^3),\qquad
 Q>0,\qquad Q'(r)<0\quad(r>0),
\]
\[
 |Q(x)|\le C(1+|x|)^{-4},\qquad
 |\nabla Q(x)|\le C(1+|x|)^{-5}.
\]
For the gradient bound one applies the radial version of Lemma~3.3 in
that reference.  Since $Q^2\in L^1\cap L^\infty$, $V=|x|^{-1}*Q^2$
is bounded and tends to zero.  The exchange operator has kernel
$2Q(x)Q(y)/|x-y|$, and
\[
 \iint\frac{4Q(x)^2Q(y)^2}{|x-y|^2}\,\dd x\,\dd y
 \le C\|Q\|_3^4<\infty.
\]

The exchange term is therefore Hilbert-Schmidt on $L^2$.
Multiplication by $V$ is compact from $H^s$ to $L^2$ for $s=1/2,1$.
Indeed, if $h_j\rightharpoonup0$ in $H^s$, local Rellich compactness and
$V(x)\to0$ yield
\begin{align*}
 \|Vh_j\|_2
 &\le\|V\|_\infty\|h_j\|_{L^2(\{|x|\le R\})}
       +\|V\|_{L^\infty(\{|x|>R\})}\|h_j\|_2,\\
 \limsup_{j\to\infty}\|Vh_j\|_2
 &\le C\|V\|_{L^\infty(\{|x|>R\})}\longrightarrow0
       \quad(R\to\infty).
\end{align*}
In particular, $M_V(\sqrt{-\Delta}+1)^{-1}$ is compact on $L^2$.
Bounded self-adjoint perturbation and Weyl's theorem give
\[
 \mathcal D(L_+)=\mathcal D(L_-)=H^1(\mathbb R^3),
\]
and

\begin{equation}\label{7.1}
 \sigma_{\rm ess}(L_+)=\sigma_{\rm ess}(L_-)=[1,\infty).
\end{equation}

\begin{lemma}\label{7.2}
One has $\frac32Q+x\cdot\nabla Q\in H^1(\mathbb R^3)$ and
\begin{equation}\label{7.3}
 L_+\left(\frac32Q+x\cdot\nabla Q\right)=-Q.
\end{equation}
Consequently, for $f_*\in X_A$ from \eqref{6.12} and its
representative $g=f_0+w$, both $\Lambda f_*$ and $\Lambda g$ belong to $X_A$.
Their dilation orbits are differentiable in $X_A$.
\end{lemma}
\begin{proof}
The displayed decay estimates imply
$\frac32Q+x\cdot\nabla Q\in L^2$.  If
$Q_\lambda(x)=\lambda^{3/2}Q(\lambda x)$, then
\[
 (\sqrt{-\Delta}+\lambda)Q_\lambda
 =\bigl(|x|^{-1}*Q_\lambda^2\bigr)Q_\lambda.
\]
For $\lambda$ near $1$, the decay bounds imply
\[
 |Q_\lambda(x)|+|\partial_\lambda Q_\lambda(x)|
 \le C(1+|x|)^{-4}.
\]
They justify differentiation of the Hartree integrals.  Differentiation
in distributions at $\lambda=1$ gives \eqref{7.3}.  All the terms other than
$\sqrt{-\Delta}(\frac32Q+x\cdot\nabla Q)$ belong to $L^2$, since
$V$ and the exchange operator are bounded on $L^2$.  Fourier multiplication
then gives $H^1$ regularity.  Radial reduction and dilation give the
assertions for $f_*$ and $g$.  Finally, since $\Lambda f\in L^2$, the identity
\[
 f_\lambda-f=\int_1^\lambda\frac1t(\Lambda f)_t\,\dd t.
\]
holds in $L^2$.  When $\Lambda f\in X_A$, its right-hand side is a Bochner
integral in $X_A$.  Strong continuity of dilation then proves
$X_A$ differentiability.
\end{proof}

Uniqueness alone does not determine a Hessian kernel.
We first identify the two zero directions of the scale-invariant
quotient.  The fixed-frequency correction then removes dilation; the
sphere constraint removes amplitude.

\begin{lemma}\label{7.4}
For every real $k\in X_A$,
\begin{equation}\label{7.5}
\mathrm D^2\mathcal B(f_*)[k,k]\le 0,
\end{equation}
with equality if and only if
$k\in\operatorname{span}\{f_*,\Lambda f_*\}$.
\end{lemma}
\begin{proof}
Take the representative $g=f_0+w$ of $f_*$ from
\cref{5.1}.  By \cref{6.5},
$A[w]^{1/2}<1/20$.  Every amplitude or dilation of $g$ is also a maximizer, so its first
variation vanishes.  Differentiating these vanishing first variations,
using \cref{7.2}, gives
\[
 \mathrm D\mathcal B(g)=0,\qquad
 \mathrm D^2\mathcal B(g)[g,\cdot]=0,\qquad
 \mathrm D^2\mathcal B(g)[\Lambda g,\cdot]=0.
\]
The two symmetry directions are transverse to $Y_2$.  Indeed,
\[
 \langle g,\phi_0\rangle=\frac1{\sqrt2},\qquad
 \langle g,\phi_1\rangle=0,\qquad
 \langle\Lambda g,\phi_0\rangle=0,
\]
and, using \eqref{4.6},
\begin{align*}
 \langle\Lambda g,\phi_1\rangle
 &=\sqrt{\frac38}-\langle w,\Lambda\phi_1\rangle\\
 &\ge\sqrt{\frac38}-\frac{\sqrt{11}}{40}>0.
\end{align*}
Here $\Lambda$ is skew-adjoint in the $L^2$ pairing on its domain.
For $k\in X_A$, set
\[
 a=\sqrt2\langle k,\phi_0\rangle,\qquad
 b=\frac{\langle k,\phi_1\rangle}{\langle\Lambda g,\phi_1\rangle},
 \qquad k_0=k-ag-b\Lambda g.
\]
The preceding pairings give
$\langle k_0,\phi_0\rangle=\langle k_0,\phi_1\rangle=0$, so $k_0\in Y_2$.
Thus
\begin{align*}
 \mathrm D^2\mathcal B(g)[k,k]
 &=\mathrm D^2\mathcal B(g)[k_0,k_0]
 =\mathrm D^2\widetilde\Psi(w)[k_0,k_0]\le-\frac74A[k_0].
\end{align*}
Equality therefore holds exactly on $\operatorname{span}\{g,\Lambda g\}$.
The fixed bounded linear isomorphism $f\mapsto af_\lambda$ taking $f_* $
to $g$ preserves $\mathcal B$ and commutes with $\Lambda$.
Congruence of the Hessians proves \eqref{7.5}.
\end{proof}

\begin{lemma}\label{7.6}
Let $h\in H^{1/2}(\mathbb R^3;\mathbb R)$ satisfy
\begin{equation}\label{7.7}
 \langle(\sqrt{-\Delta}+1)Q,h\rangle=0.
\end{equation}
Then $\langle h,L_+h\rangle\ge0$.  If $h\ne0$ is radial, the inequality
is strict.  In particular,
\[
 \ker L_+\cap L^2_{\rm rad}(\mathbb R^3)=\{0\}.
\]
\end{lemma}
\begin{proof}
By \eqref{6.13}, $Q$ maximizes $\mathcal H$ on
$\mathcal K(u)=\mathcal K(Q)=\mathcal H(Q)$.  For $h$ satisfying
\eqref{7.7}, the curve
\[
 \gamma(t)=\left(\frac{\mathcal H(Q)}
 {\mathcal H(Q)+t^2\mathcal K(h)}\right)^{1/2}(Q+th)
\]
lies on this sphere.  Direct differentiation yields
\[
 0\ge\left.\frac{\dd^2}{\dd t^2}\right|_{t=0}\mathcal H(\gamma(t))
 =-4\langle h,L_+h\rangle.
\]

Suppose now that $h$ is radial and set $k(r)=rh(r)$.  Since
$rQ(r)=\beta_*^{-1/2}f_*(r)$, condition
\eqref{7.7} becomes $A[f_*,k]=0$.
The identity
\[
 \frac{\mathfrak b(f)}{A[f]^2}
 =\mathcal B(f)\left[1-
 \left(\frac{T(f)-M(f)}{A[f]}\right)^2\right]
\]
and $M(f_*)=T(f_*)=1/2$ give
\begin{equation}\label{7.8}
 \left.\frac{\mathrm d^2}{\mathrm dt^2}\right|_{t=0}
 \frac{\mathfrak b(f_*+tk)}{A[f_*+tk]^2}
 =\mathrm D^2\mathcal B(f_*)[k,k]
  -8\beta_*\langle(D-\Id)f_*,k\rangle^2.
\end{equation}
Both terms on the right are nonpositive.  If their sum vanishes,
\cref{7.4} implies
$k=af_*+b\Lambda f_*$.  Differentiating $M((f_*)_\lambda)$ and
$T((f_*)_\lambda)$ at $\lambda=1$ yields
\[
 \langle f_*,\Lambda f_*\rangle=0,\qquad
 \langle Df_*,\Lambda f_*\rangle=\frac14.
\]
Hence
\[
 \langle(D-\Id)f_*,k\rangle=\frac b4=0,\qquad
 A[f_*,k]=a+\frac b4=0,
\]
so $k=0$.  Thus the left-hand side of
\eqref{7.8} is strictly negative for every
nonzero $k\in f_*^{\perp_A}$.
The radial identities in \cref{2.1} give
\begin{align*}
 \left.\frac{\mathrm d^2}{\mathrm dt^2}\right|_{t=0}
 \frac{\mathfrak b(f_*+tk)}{A[f_*+tk]^2}
 &=4\mathfrak C(f_*^2,k^2)+8\mathfrak C(f_*k,f_*k)-4\beta_*A[k]\\
 &=-\frac{\beta_*}{\pi}\langle h,L_+h\rangle.
\end{align*}
This proves strict positivity.  Finally, $L_+Q=-2(\sqrt{-\Delta}+1)Q$.
For a radial zero mode, self-adjointness therefore gives
\eqref{7.7}, and the strict inequality forces
$h=0$.
\end{proof}

\begin{proof}[Proof of \cref{1.3}]
Differentiating the equation gives
\[
 L_+\partial_{x_j}Q=0,\qquad j=1,2,3.
\]
Let $\{Y_{\ell m}:-\ell\le m\le\ell,\ \ell=0,1,\ldots\}$ be an
orthonormal real spherical-harmonic basis.  For $\ell\ge0$, let $D_\ell$
be the nonnegative square root of the Friedrichs realization of
\[
 -\partial_r^2-\frac2r\partial_r+\frac{\ell(\ell+1)}{r^2}
 \quad\text{on }L^2(\mathbb R_+,r^2\dd r).
\]
The map $\varphi\mapsto r\varphi$ is unitary to $L^2(\mathbb R_+)$, and
\begin{align*}
 \langle\varphi,D_\ell^2\varphi\rangle
 &=\int_0^\infty\left(r^2|\varphi'|^2
                         +\ell(\ell+1)|\varphi|^2\right)\dd r\\
 &=\int_0^\infty\left(|(r\varphi)'|^2
                +\frac{\ell(\ell+1)}{r^2}|r\varphi|^2\right)\dd r.
\end{align*}
Here $H^1_0(\mathbb R_+)$ is the $H^1$ closure of
$C_c^\infty(\mathbb R_+)$.  The one-dimensional Hardy inequality
\[
 \int_0^\infty\frac{|u|^2}{r^2}\,\dd r
 \le4\int_0^\infty|u'|^2\,\dd r,\qquad u\in H^1_0(\mathbb R_+),
\]
shows that these second-order forms have the common domain
$\{\varphi:r\varphi\in H^1_0(\mathbb R_+)\}$ and equivalent form norms
for each fixed $\ell$.  Thus, in form sense,
\[
 D_\ell^2\succeq D_1^2\quad(\ell\ge1),\qquad
 D_\ell\succeq D_1
\]
by the L\"owner-Heinz inequality.
With $r_<:=\min\{r,s\}$ and $r_>:=\max\{r,s\}$, the radial-factor
operator in the $\ell$ sector is
\begin{equation}\label{7.9}
\begin{split}
 (L_{+,\ell}\varphi)(r)
 ={}&D_\ell\varphi(r)+(1-V(r))\varphi(r)\\
 &-\frac{8\pi}{2\ell+1}Q(r)
 \int_0^\infty\frac{r_<^\ell}{r_>^{\ell+1}}
 Q(s)\varphi(s)s^2\,\dd s.
\end{split}
\end{equation}
Every sector with $\ell\ge1$ is orthogonal to the radial function
$(\sqrt{-\Delta}+1)Q$.  Thus \cref{7.6} gives
\[
 L_{+,\ell}\succeq0,\qquad \ell\ge1.
\]
The same compact-perturbation argument as above gives
$\sigma_{\mathrm{ess}}(L_{+,\ell})=[1,\infty)$.
The heat semigroup $e^{-uD_\ell^2}$ is positivity improving: after the
unitary map $\varphi\mapsto r\varphi$, this is the Dirichlet
Schr\"odinger semigroup with the locally bounded nonnegative potential
$\ell(\ell+1)/r^2$ on the connected interval $(0,\infty)$.
Equivalently, its Feynman-Kac kernel is strictly positive on every
compact subinterval.  See also \cite[Section~7]{Lenzmann2009}.
Subordination gives
\[
 e^{-tD_\ell}=\frac{t}{2\sqrt\pi}\int_0^\infty
 u^{-3/2}e^{-t^2/(4u)}e^{-uD_\ell^2}\,\dd u.
\]
The two terms subtracted from $D_\ell+1$ in \eqref{7.9} are bounded and
positivity preserving.  The Dyson expansion therefore implies
\[
 e^{-tL_{+,\ell}}\varphi
 \ge e^{-t(D_\ell+1)}\varphi>0
 \quad(t>0,\ \varphi\ge0,\ \varphi\ne0).
\]
The Perron-Frobenius theorem now makes an isolated lowest eigenvalue
simple, with a strictly positive eigenfunction.
For $\ell=1$, the strictly positive function $-Q'$ is a zero mode; hence
\[
 \ker L_{+,1}=\operatorname{span}\{Q'\}.
\]

Suppose $\ell\ge2$ admits a zero mode.  It can be chosen strictly positive,
say $\varphi>0$.  The form ordering already proved gives
$\varphi\in\mathcal D(D_1^{1/2})$ and
$\langle\varphi,D_\ell\varphi\rangle\ge\langle\varphi,D_1\varphi\rangle$.
Therefore,
\begin{align*}
 0=\langle\varphi,L_{+,\ell}\varphi\rangle
 \ge{}&\langle\varphi,L_{+,1}\varphi\rangle\\
 &+8\pi\iint Q(r)\varphi(r)Q(s)\varphi(s)
 \left(\frac{r_<}{3r_>^2}
 -\frac{r_<^\ell}{(2\ell+1)r_>^{\ell+1}}\right)
 r^2s^2\,\dd r\,\dd s.
\end{align*}
The first term is nonnegative.  For $t=r_</r_>\in(0,1]$,
\[
 \frac t3-\frac{t^\ell}{2\ell+1}
 \ge\left(\frac13-\frac15\right)t>0.
\]
Thus the double integral is strictly positive, a contradiction.
The radial sector is trivial by \cref{7.6}; consequently,
\[
 \ker L_+=\operatorname{span}_{\mathbb R}
 \{\partial_{x_1}Q,\partial_{x_2}Q,\partial_{x_3}Q\}.
\]

Finally, $L_-Q=0$.  The free Poisson kernel is strictly positive and
$V\ge0$ is bounded; hence the same Dyson argument gives
\[
 e^{-tL_-}h\ge e^{-t(\sqrt{-\Delta}+1)}h>0
 \quad(t>0,\ h\ge0,\ h\ne0).
\]
 If its lowest
spectral value were negative, \eqref{7.1} would give
a strictly positive eigenfunction orthogonal to $Q>0$, which is impossible.
Thus $L_-\succeq0$, and simplicity of its lowest eigenvalue gives
\[
 \ker L_-=\operatorname{span}_{\mathbb R}\{Q\}.
\]
\end{proof}

\section{The Lieb-Yau conjecture near the critical mass}
\label{8.0}

By \cite{GuoZeng2017}, every centered positive minimizer with
\(N\to N_*\), after rescaling, converges to a critical optimizer.  Thus
\cref{1.4} follows from radial local uniqueness near \(Q\)
and strict monotonicity of the mass along the local branch.

\begin{lemma}\label{8.1}
For every \(u\in H^{1/2}(\mathbb R^3)\),
\begin{equation*}
 \mathcal H(u)\leq\frac{2}{N_*}M(u)T(u).
\end{equation*}
Equality holds for a nonzero \(u\) if and only if
\begin{equation}\label{8.2}
 u(x)=c\lambda^{3/2}Q\bigl(\lambda(x-x_0)\bigr),
 \qquad
 c\in\mathbb C\setminus\{0\},\quad \lambda>0,\quad x_0\in\mathbb R^3.
\end{equation}
Moreover,
\begin{equation}\label{8.3}
 M(Q)=T(Q)=N_*,\qquad \mathcal H(Q)=2N_*.
\end{equation}
\end{lemma}

The sharp inequality, the existence of optimizers, and
\eqref{8.3} are standard; see
\cite[Appendix~A.2]{LiebYau1987} and
\cite[Appendix~A]{LenzmannLewin2011}.  The optimizer ground state
correspondence in \cite[Section~2.1]{FrankLenzmann2010}, together with
\cref{1.2}, gives
\eqref{8.2}.

For \(m>0\),
\begin{equation}\label{8.4}
 \mathcal E_m\bigl(m^{3/2}v(m\,\cdot)\bigr)
 =m\mathcal E_1(v),\qquad
 e_m(N)=me_1(N).
\end{equation}
Hence it suffices to consider $m=1$.  By
\cite[Theorem~1(iii)]{Lenzmann2009}, every positive minimizer can be
translated to a radial nonincreasing function centered at the origin.
Let $u_N$ be any such minimizer of $e_1(N)$ and let $\varrho_N$ be its
Lagrange multiplier, with the sign convention
\begin{equation*}
 \bigl(\sqrt{-\Delta+1}-1+\varrho_N\bigr)u_N
 =\bigl(|x|^{-1}*u_N^2\bigr)u_N.
\end{equation*}
For \(N\) sufficiently close to \(N_*\), set
\begin{equation*}
 \eta_N:=\varrho_N-1>0,\qquad
 \tau_N:=\eta_N^{-2},\qquad
 w_N(x):=\eta_N^{-3/2}u_N(x/\eta_N).
\end{equation*}
Then
\begin{equation*}
 \bigl(\sqrt{-\Delta+\tau_N}+1\bigr)w_N
 =\bigl(|x|^{-1}*w_N^2\bigr)w_N,
\end{equation*}
The concentration theorem
\cite[Theorem~1.3 and Proposition~1.4]{GuoZeng2017}, with interaction
coefficient identically $1$, and \cref{1.2} imply
\begin{equation}\label{8.5}
 \tau_N\to0,\qquad w_N\to Q\quad\text{strongly in }H^{1/2}(\mathbb R^3)
 \qquad(N\to N_*).
\end{equation}
We give the normalization and centering details.  Set
\[
 \lambda_N:=\frac{T(u_N)}{N_*}\to\infty,\qquad
 v_N(x):=\lambda_N^{-3/2}u_N(x/\lambda_N),\qquad
 M(v_N)=N,\quad T(v_N)=N_*.
\]
For any sequence $N_j\to N_*$ and any corresponding minimizers, the
quoted theorem gives, after extraction, translations $y_j$ such that
$v_{N_j}(\,\cdot+y_j)$ converges strongly in $H^{1/2}$ to a nonzero
critical optimizer.  In particular, for some $R>0$,
\[
 \liminf_{j\to\infty}\int_{|x-y_j|<R}|v_{N_j}(x)|^2\,\dd x>0.
\]
Radial monotonicity gives
\[
 \frac{4\pi r^3}{3}|v_{N_j}(r)|^2
 \le\int_{|x|<r}|v_{N_j}(x)|^2\,\dd x\le N_j.
\]
If $|y_j|\to\infty$, then
\[
 \int_{|x-y_j|<R}|v_{N_j}(x)|^2\,\dd x
 \le\frac{R^3N_j}{(|y_j|-R)^3}\longrightarrow0,
\]
a contradiction.  Thus $y_j$ is bounded.  Strong continuity of
translation yields convergence of the unshifted profiles, whose limit
is radial about the origin.  By \eqref{8.2}, it has the form
$a\lambda^{3/2}Q(\lambda\,\cdot)$, $a,\lambda>0$; the limits of $M$ and
$T$ give
\[
 a^2N_*=N_*,\qquad a^2\lambda N_*=N_*,
 \qquad a=\lambda=1.
\]
Therefore $v_{N_j}\to Q$.  The equation for $v_N$ is
\[
 \left(\sqrt{-\Delta+\lambda_N^{-2}}
          +\frac{\eta_N}{\lambda_N}\right)v_N
 =\bigl(|x|^{-1}*v_N^2\bigr)v_N.
\]
Pairing with $v_N$ and using
$0\le\sqrt{t^2+\lambda_N^{-2}}-t\le\lambda_N^{-1}$ gives
\[
 \frac{\eta_N}{\lambda_N}
 =\frac{\mathcal H(v_N)
 -\langle v_N,\sqrt{-\Delta+\lambda_N^{-2}}\,v_N\rangle}{N}
 \longrightarrow\frac{2N_*-N_*}{N_*}=1.
\]
In particular, $\eta_N>0$ for $N$ near $N_*$, and
\[
 w_N(x)=\left(\frac{\lambda_N}{\eta_N}\right)^{3/2}
 v_N\!\left(\frac{\lambda_N}{\eta_N}x\right)\longrightarrow Q.
\]
This proves \eqref{8.5}.  The argument applies to every sequence of
minimizers, not merely to a preselected family.

For \(\tau\geq0\), set
\[
 \mathscr A_\tau:=\sqrt{-\Delta+\tau}+1,
 \qquad
 \mathcal G(w):=\bigl(|x|^{-1}*w^2\bigr)w.
\]
The standard Hartree estimates \cite{Lenzmann2007} give
\[
 \mathcal G\in C^\infty\!\left(
 H^{1/2}_{\rm rad}(\mathbb R^3;\mathbb R),
 H^{-1/2}_{\rm rad}(\mathbb R^3;\mathbb R)\right).
\]
$\mathscr A_0:H^{1/2}_{\rm rad}\to H^{-1/2}_{\rm rad}$
is an isomorphism.
Moreover,
\[
 \mathcal G'(Q)h=Vh+2Q\bigl(|x|^{-1}*(Qh)\bigr).
\]
The first term is compact from $H^{1/2}$ to $L^2$ and the second is compact on $L^2$ by its
Hilbert-Schmidt kernel.  Consequently,
\[
 \mathscr A_0^{-1}\mathcal G'(Q):H^{1/2}_{\rm rad}\to H^{1/2}_{\rm rad}
 \quad\text{is compact},\qquad
 L_+=\mathscr A_0\bigl(\Id-\mathscr A_0^{-1}\mathcal G'(Q)\bigr).
\]
Thus $L_+:H^{1/2}_{\rm rad}\to H^{-1/2}_{\rm rad}$ is Fredholm of index
zero.  A weak radial zero mode satisfies
\[
 \mathscr A_0h=\mathcal G'(Q)h\in L^2
 \quad\Longrightarrow\quad h\in H^1_{\rm rad}
 \quad\Longrightarrow\quad h=0
\]
by \cref{7.6}.  The Fredholm alternative therefore gives
\begin{equation}\label{8.6}
 L_+=\mathscr A_0-\mathcal G'(Q):
 H^{1/2}_{\rm rad}(\mathbb R^3;\mathbb R)
 \xrightarrow{\ \sim\ }
 H^{-1/2}_{\rm rad}(\mathbb R^3;\mathbb R).
\end{equation}
Throughout this section, $L_+^{-1}$ denotes the inverse of this
radial restriction, not an inverse on the full nonradial space.

\begin{lemma}\label{8.7}
There exist \(\tau_0,r_0>0\) such that, for every
\(0\leq\tau\leq\tau_0\), the equation
\begin{equation}\label{8.8}
 \mathscr A_\tau w=\mathcal G(w)
\end{equation}
has exactly one solution
\(w_\tau\in H^{1/2}_{\rm rad}(\mathbb R^3;\mathbb R)\) with
\(\|w_\tau-Q\|_{H^{1/2}_{\rm rad}}<r_0\).  Moreover,
\[
 w_\tau\to Q\quad\text{in }H^{1/2}_{\rm rad}(\mathbb R^3;\mathbb R)
 \quad(\tau\to0),
\]
the map \(\tau\mapsto w_\tau\) is \(C^1\) on \((0,\tau_0]\), and,
after decreasing \(\tau_0\),
\begin{equation}\label{8.9}
 \frac{\dd}{\dd\tau}\|w_\tau\|_2^2<0,
 \qquad 0<\tau\leq\tau_0.
\end{equation}
\end{lemma}

\begin{proof}
At $\tau=0$ the operator depends only $1/2$-H\"older continuously on
the parameter in operator norm.  We therefore construct the branch by
contraction for $\tau\ge0$ and differentiate only for $\tau>0$.
The exact Hartree remainder is
\begin{align*}
 \mathcal G(Q+z)-\mathcal G(Q)-\mathcal G'(Q)z
 ={}&2\bigl(|x|^{-1}*(Qz)\bigr)z\\
 &+\bigl(|x|^{-1}*z^2\bigr)Q
   +\bigl(|x|^{-1}*z^2\bigr)z.
\end{align*}
The continuous trilinear Hartree estimate gives, for $\|z_j\|_{H^{1/2}}\le1$,
\begin{align*}
 &\|\mathcal G(Q+z_1)-\mathcal G(Q+z_2)
                   -\mathcal G'(Q)(z_1-z_2)\|_{H^{-1/2}}\\
 &\qquad\le C\bigl(\|z_1\|_{H^{1/2}}+\|z_2\|_{H^{1/2}}\bigr)
                   \|z_1-z_2\|_{H^{1/2}}.
\end{align*}
For \(z\in H^{1/2}_{\rm rad}(\mathbb R^3;\mathbb R)\) with
\(\|z\|_{H^{1/2}_{\rm rad}}\leq1\),
\begin{equation}\label{8.10}
 \|\mathcal G(Q+z)-\mathcal G(Q)-\mathcal G'(Q)z\|_{H^{-1/2}_{\rm rad}}
 \leq C\|z\|_{H^{1/2}_{\rm rad}}^2,
\end{equation}
and
\begin{equation}\label{8.11}
 \|\mathscr A_\tau-\mathscr A_0\|_{
 H^{1/2}_{\rm rad}\to H^{-1/2}_{\rm rad}}\leq\sqrt\tau,
 \qquad
 \|(\mathscr A_\tau-\mathscr A_0)Q\|_{H^{-1/2}_{\rm rad}}\leq C\tau.
\end{equation}
For the last bound, $Q\in L^1\cap L^2$ gives
$|\xi|^{-1}\widehat Q\in L^2(\mathbb R^3)$, and
\[
 0\le\sqrt{|\xi|^2+\tau}-|\xi|
 =\frac{\tau}{\sqrt{|\xi|^2+\tau}+|\xi|}
 \le\frac{\tau}{2|\xi|}\quad(\xi\ne0).
\]
No bounded inverse for $\sqrt{-\Delta}$ on all of $L^2$ is used here.
Writing \(w=Q+z\),
\eqref{8.8} is equivalent to
\[
 z=L_+^{-1}\!\left[
 \mathcal G(Q+z)-\mathcal G(Q)-\mathcal G'(Q)z
 -(\mathscr A_\tau-\mathscr A_0)(Q+z)
 \right].
\]
By \eqref{8.6}, \eqref{8.10}, and \eqref{8.11}, the norm of the right-hand side is bounded by
\[
 C\left(\|z\|_{H^{1/2}}^2+\sqrt\tau\|z\|_{H^{1/2}}+\tau\right),
\]
and its Lipschitz constant on $\|z\|_{H^{1/2}}\le C_0\tau$ is at most
$C(2C_0\tau+\sqrt\tau)$.  Choose $C_0>2C$ and then $\tau_0>0$ so that
\[
 C(C_0^2\tau_0+C_0\sqrt{\tau_0}+1)\le C_0,\qquad
 C(2C_0\tau_0+\sqrt{\tau_0})<\frac12.
\]
The closed ball $\|z\|_{H^{1/2}_{\rm rad}}\le C_0\tau$ is invariant and
the map is a contraction for $0<\tau\le\tau_0$.  At $\tau=0$, take
$w_0=Q$.  Thus
\begin{equation}\label{8.12}
 \|w_\tau-Q\|_{H^{1/2}_{\rm rad}}\leq C\tau.
\end{equation}
If \(w_1,w_2\in H^{1/2}_{\rm rad}(\mathbb R^3;\mathbb R)\),
\(\|w_j-Q\|_{H^{1/2}_{\rm rad}}<r_0\), solve
\eqref{8.8}, then
\[
 \left[\mathscr A_\tau-
 \int_0^1\mathcal G'\bigl(w_2+s(w_1-w_2)\bigr)\,\dd s\right](w_1-w_2)=0.
\]
Indeed,
\begin{align*}
 &\left\|L_+^{-1}\left[
 \mathscr A_\tau-\int_0^1\mathcal G'(w_2+s(w_1-w_2))\,\dd s-L_+
 \right]\right\|_{H^{1/2}_{\rm rad}\to H^{1/2}_{\rm rad}}\\
 &\qquad\le C(\sqrt\tau+r_0)<\frac12
\end{align*}
when $r_0,\tau_0$ are small.  A Neumann series proves invertibility,
so $w_1=w_2$ in this fixed neighborhood.  Decreasing $\tau_0$ also ensures
$C_0\tau_0<r_0$.  For $\tau>0$ the operator family is $C^1$, with
\[
 \partial_\tau\mathscr A_\tau=\frac12(-\Delta+\tau)^{-1/2}.
\]
The implicit-function theorem then yields
\[
 w_\tau\in C^1\!\left((0,\tau_0];
 H^{1/2}_{\rm rad}(\mathbb R^3;\mathbb R)\right).
\]
Write $w_\tau':=\partial_\tau w_\tau$. Differentiation gives
\begin{equation}\label{8.13}
 \bigl(\mathscr A_\tau-\mathcal G'(w_\tau)\bigr)w_\tau'
 =-\frac12(-\Delta+\tau)^{-1/2}w_\tau.
\end{equation}
Moreover,
\[
 \|\mathscr A_\tau-\mathcal G'(w_\tau)-L_+\|_{
 H^{1/2}_{\rm rad}\to H^{-1/2}_{\rm rad}}=O(\sqrt\tau),
\]
and, by \eqref{8.12},
\[
 \|(-\Delta+\tau)^{-1/2}(w_\tau-Q)\|_{H^{-1/2}_{\rm rad}}
 \leq \tau^{-1/2}\|w_\tau-Q\|_{H^{1/2}_{\rm rad}}
 =O(\sqrt\tau).
\]
Since \(Q\in L^1(\mathbb R^3)\cap L^2(\mathbb R^3)\),
\[
 \int_{\mathbb R^3}|\xi|^{-2}|\widehat Q(\xi)|^2\,\dd\xi
 \leq
 \|\widehat Q\|_\infty^2
 \int_{|\xi|\leq1}|\xi|^{-2}\,\dd\xi
 +\|Q\|_2^2<\infty.
\]
With the unitary Fourier normalization,
\[
 \widehat{\frac1{2\pi^2}\bigl(|x|^{-2}*Q\bigr)}(\xi)
 =|\xi|^{-1}\widehat Q(\xi).
\]
Therefore
\begin{align*}
 &\left\|(-\Delta+\tau)^{-1/2}Q
 -\frac1{2\pi^2}\bigl(|x|^{-2}*Q\bigr)\right\|_{H^{-1/2}}^2
 \\
 &\quad=
 \int_{\mathbb R^3}(1+|\xi|^2)^{-1/2}
 \left((|\xi|^2+\tau)^{-1/2}-|\xi|^{-1}\right)^2
 |\widehat Q(\xi)|^2\,\dd\xi
 \longrightarrow0.
\end{align*}
It follows from \eqref{8.13} that
\begin{equation}\label{8.14}
 w_\tau'\longrightarrow
 -\frac12L_+^{-1}
 \left[\frac1{2\pi^2}\bigl(|x|^{-2}*Q\bigr)\right]
 \quad\text{in }H^{1/2}_{\rm rad}(\mathbb R^3;\mathbb R).
\end{equation}
By \cref{7.2}, the dilation identity
\eqref{7.3} holds in $L^2$.
Moreover,
\[
 \left\langle \lambda^{3/2}Q(\lambda\,\cdot),
 \frac1{2\pi^2}|x|^{-2}*
 \bigl[\lambda^{3/2}Q(\lambda\,\cdot)\bigr]\right\rangle
 =\lambda^{-1}
 \left\langle Q,\frac1{2\pi^2}\bigl(|x|^{-2}*Q\bigr)\right\rangle,
\]
so
\[
 2\left\langle \frac32Q+x\cdot\nabla Q,
 \frac1{2\pi^2}\bigl(|x|^{-2}*Q\bigr)\right\rangle
 =-\left\langle Q,
 \frac1{2\pi^2}\bigl(|x|^{-2}*Q\bigr)\right\rangle.
\]
Hence, by self-adjointness of \(L_+^{-1}\),
\begin{equation}\label{8.15}
 \left\langle Q,L_+^{-1}
 \left[\frac1{2\pi^2}\bigl(|x|^{-2}*Q\bigr)\right]\right\rangle
 =\frac12\left\langle Q,
 \frac1{2\pi^2}\bigl(|x|^{-2}*Q\bigr)\right\rangle>0.
\end{equation}
Combining \eqref{8.14} and
\eqref{8.15},
\[
 \frac{\dd}{\dd\tau}\|w_\tau\|_2^2
 \longrightarrow
 -\frac12\left\langle Q,
 \frac1{2\pi^2}\bigl(|x|^{-2}*Q\bigr)\right\rangle<0.
\]
\end{proof}

\begin{proof}[Proof of Corollary~\ref{1.4}]
There is a single $0<\delta_*<N_*$ such that, for every positive
radial minimizer $u_N$ with $N_*-\delta_*<N<N_*$,
\[
 0<\tau_N<\tau_0,\qquad \|w_N-Q\|_{H^{1/2}}<r_0.
\]
Otherwise, one could choose $N_j\to N_*$ and corresponding minimizers
such that
\[
 \tau_{N_j}\ge\tau_0
 \quad\text{or}\quad
 \|w_{N_j}-Q\|_{H^{1/2}}\ge r_0.
\]
This contradicts the sequential conclusion \eqref{8.5}.
By \cref{8.7}, every such $w_N$ equals $w_{\tau_N}$.
If two minimizers have mass $N$, then
\[
 \|w_{\tau_1}\|_2^2=N=\|w_{\tau_2}\|_2^2.
\]
Strict monotonicity \eqref{8.9} gives $\tau_1=\tau_2$.
Local uniqueness and the inverse rescaling
\[
 \eta_N=\tau_N^{-1/2},\qquad
 u_N(x)=\eta_N^{3/2}w_{\tau_N}(\eta_Nx)
\]
give equality of the centered minimizers.  Finally, \eqref{8.4} preserves
$N$ and bijects minimizers for $m=1$ and $m>0$; hence the same
$\delta_*$ works for every rest mass.  Translation invariance gives the
stated uniqueness up to translations.
\end{proof}

\noindent \textbf{Acknoledgments}: The research of Juncheng Wei was supported by National R\&D Program of China (Grant
No. 2022YFA1005602), and Hong Kong General Research Fund (No. 14303125) “On Fujita
equation in the critical or supercritical regime”.  Yuanyang Yu was supported by National Natural
Science Foundation of China (No.12301136). We thank Louise Gassot and Nicolas Camps for introducing us to the Benjamin-Ono function.

\bigskip

\noindent \textbf{Declarations of interest}: None.

\bigskip
\noindent \textbf{Data availability statement}: There are no new data associated with this article.

\bigskip
 \noindent \textbf{AI assistance statement}: The authors used AI tools to assist with exploring, coding and manuscript editing; all mathematical validation, final proof decisions, and final wording remain the sole responsibility of the human authors.

\clearpage 
 \phantomsection

\newpage
\noindent {Pan Chen\\
School of Mathematical Sciences,\\
 Shanghai Jiao Tong University, Shanghai 200240, P.R. China
\\
e-mail: chenpan2020@amss.ac.cn }
\medskip
\\
\noindent {Qi Guo\\
School of Mathematics,\\
Renmin University of China, Beijing, 100872, P.R. China\\
e-mail: qguo@ruc.edu.cn}
\medskip
\\
\noindent {Juncheng Wei\\
Department of Mathematics, \\
Chinese University of Hong Kong, Shatin, N.T., Hong Kong\\
e-mail: wei@math.cuhk.edu.hk}
\medskip
\\
\noindent {Yuanyang Yu\\
School of Mathematics and Statistics, \\
Yunnan University, Yunnan, 650500, P.R. China\\
e-mail: yyysx43@163.com }

\end{document}

%% file: boson_f0_figures_latex/fig_f0.tex
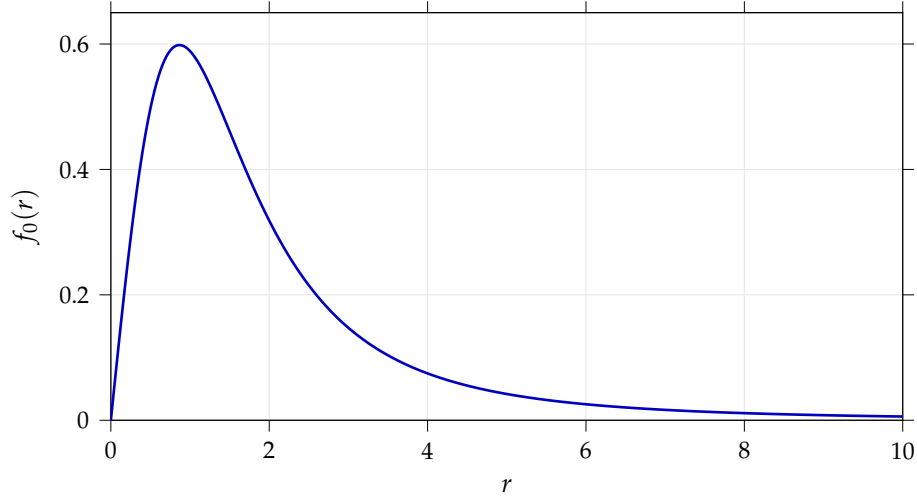
\begin{figure}[htbp]
\centering
\begingroup\normalcolor
\begin{tikzpicture}
\begin{axis}[
  width=0.68\textwidth, height=0.35\textwidth,
  scale only axis,
  domain=0:10, samples=401,
  xmin=0, xmax=10, ymin=0, ymax=0.65,
  xtick={0,2,4,6,8,10}, ytick={0,0.2,0.4,0.6},
  xlabel={$r$}, ylabel={$f_0(r)$},
  axis lines=box, tick align=outside,
  axis line style={black, line width=0.5pt},
  tick style={black},
  tick label style={font=\footnotesize, text=black},
  label style={font=\small, text=black},
  scaled ticks=false,
  grid=major, major grid style={gray!20, line width=0.3pt}
]
\addplot[blue!75!black, line width=1pt, no marks]
  {4*(3/2)^(5/2)*x/(sqrt(pi)*(x^2+9/4)^2)};
\end{axis}
\end{tikzpicture}
\caption{The reference function $f_0$.}
\label{fig:reference-profile}
\endgroup
\end{figure}

%% file: boson_f0_figures_latex/fig_f0_potentials.tex
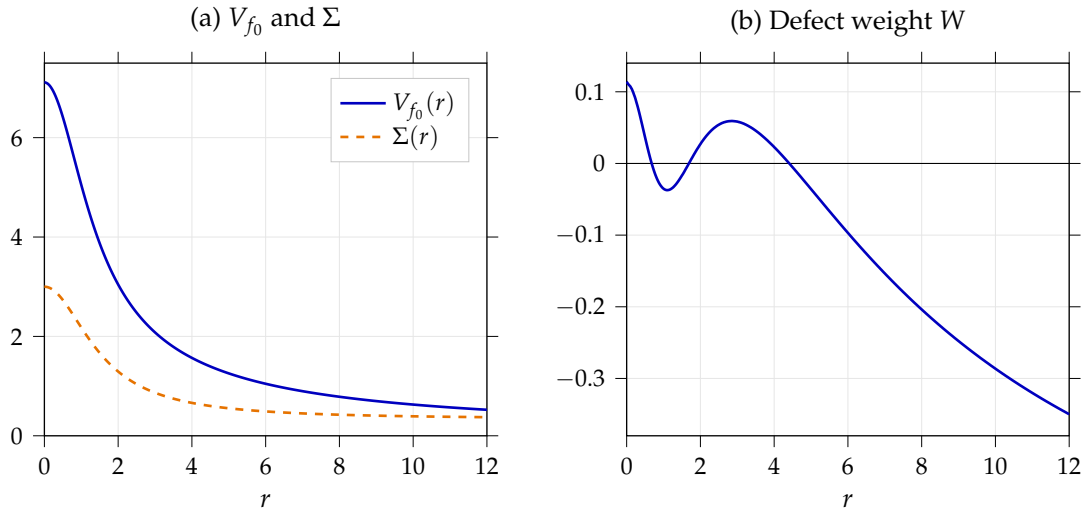
\begin{figure}[htbp]
\centering
\begingroup\normalcolor
\begin{tikzpicture}[
  declare function={
    atanzero(\t)=(\t==0 ? 1 : rad(atan(\t))/max(\t,0.00000001));
    vzero(\t)=8*atanzero(\t)/3
              +8*(3*\t^2+5)/(9*(1+\t^2)^2);
    sigzero(\t)=1/3+8/(3*(1+\t^2));
  }
]
\begin{groupplot}[
  group style={group size=2 by 1, horizontal sep=0.12\textwidth},
  width=0.38\textwidth, height=0.32\textwidth,
  scale only axis,
  domain=0:12, samples=601,
  xmin=0, xmax=12, xtick={0,2,4,6,8,10,12},
  xlabel={$r$},
  axis lines=box, tick align=outside,
  axis line style={black, line width=0.5pt},
  tick style={black},
  tick label style={font=\footnotesize, text=black},
  label style={font=\small, text=black},
  title style={font=\small, text=black},
  scaled ticks=false,
  grid=major, major grid style={gray!20, line width=0.3pt},
  legend style={
    at={(0.96,0.96)}, anchor=north east,
    font=\footnotesize, text=black,
    draw=gray!45, fill=white,
    line width=0.3pt, inner sep=3pt,
    cells={anchor=west}
  },
  every axis plot/.append style={line width=1pt, no marks}
]
\nextgroupplot[
  title={(a) $V_{f_0}$ and $\Sigma$},
  ymin=0, ymax=7.5, ytick={0,2,4,6}
]
\addplot[blue!75!black] {vzero(2*x/3)};
\addlegendentry{$V_{f_0}(r)$}
\addplot[orange!90!black, dashed] {sigzero(2*x/3)};
\addlegendentry{$\Sigma(r)$}

\nextgroupplot[
  title={(b) Defect weight $W$},
  ymin=-0.38, ymax=0.14, ytick={-0.3,-0.2,-0.1,0,0.1}
]
\addplot[black, line width=0.4pt, forget plot]
  coordinates {(0,0) (12,0)};
\addplot[blue!75!black] {vzero(2*x/3)-(7/3)*sigzero(2*x/3)};
\end{groupplot}
\end{tikzpicture}
\caption{Left: $V_{f_0}$ and $\Sigma=Af_0/f_0$.
Right: the defect weight $W=V_{f_0}-\frac73\Sigma$.}
\label{fig:f0-potentials}
\endgroup
\end{figure}